\documentclass[12pt]{article}
\usepackage[top=1in, bottom=1in, left=1in, right=1in]{geometry}
\usepackage{amsmath}
\usepackage{enumerate}
\usepackage{float}
\usepackage{fancyvrb}
\usepackage{amssymb}
\usepackage{color}
\usepackage{graphicx}
\usepackage{mathrsfs}
\usepackage[ruled,linesnumbered]{algorithm2e}

\usepackage{mathabx}
\usepackage{amsthm}
\usepackage[dvipsnames]{xcolor}
\usepackage{wasysym}
\usepackage{tikz}
\usepackage{subfig}
\usepackage[font=small]{caption}
\usetikzlibrary{cd}
\usetikzlibrary{decorations.markings}
\usepackage{titlesec}

\titleformat*{\section}{\large\bfseries}
\newcommand{\ssection}[1]{%
  \section[#1]{\centering\normalfont\scshape #1}}

\newcommand{\funfont}[1]{{\fontfamily{cmtt}\selectfont #1}}

\theoremstyle{definition}
\newtheorem{prop}{Proposition}
\newtheorem{thm}{Theorem}
\newtheorem{lemma}{Lemma}
\newtheorem{defn}{Definition}

\newtheorem*{rmk}{Remark}

\def\S{\mathfrak{S}}

\def\Z{\mathbb{Z}}

\def\mon{\operatorname{mon}}

\def\BPD{\operatorname{BPD}}

\newcommand{\quo}[1]{\text{``}{#1}\text{''}}

\title{Bijective Proofs of Monk's rule 
for Schubert and Double Schubert Polynomials
with Bumpless Pipe Dreams }
\author{Daoji Huang}
\date{}

\SetArgSty{textnormal}
\begin{document}

\maketitle
\begin{abstract}
    We give bijective proofs of Monk's rule for
    Schubert and double Schubert polynomials computed
    with bumpless pipe dreams. In particular, 
    they specialize to bijective proofs of
    transition and cotransition
    formulas of Schubert and double Schubert 
    polynomials, which can be used to establish
    bijections with ordinary pipe dreams.
    
\end{abstract}
\ssection{Introduction}
Bumpless pipe dreams are introduced
in the context of back stable Schubert calculus by
Lam, Lee, and Shimozono \cite{lam2018back}. In that paper, the authors
introduced bumpless pipe dream polynomials and proved
that they agree with double Schubert polynomials. 
Subsequently, Weigandt  \cite{weigandt2020bumpless} expressed
Lascoux's transition formula with bumpless pipe
dream polynomials and gave a bijective proof
with bumpless pipe dreams. In a recent paper, Knutson \cite{knutson2019schubert} gave several proofs of
the cotransition formula of double Schubert polynomials,
including a combinatorial proof with ordinary pipe
dreams. Both transition and cotransition formulas
are specializations of 
(an equivalent formulation of)
Monk's rule for double
Schubert polynomials, which is an expansion
formula of the product of a linear double 
Schubert polynomial and a double Schubert polynomial.
The original Monk's rule is a
geometric version for single
Schubert polynomials, studied first in \cite{monk1959geometry}. A combinatorial proof
of it
with ordinary pipe dreams (called RC-graphs there) is given in 
\cite{bergeron1993rc}. In this paper,
we give a new bijective proof
of Monk's rule
for single Schubert polynomials
with bumpless pipe dreams, and show that
a slight modification of the
construction gives us 
a bijective proof of Monk's rule for
double Schubert polynomials using
decorated bumpless pipe dreams, which 
are bumpless pipe dreams with a binary
label on each blank tile. Combinatorial
proofs of Monk's rule for double Schubert
polynomials were not known before.
We also remark that with the
cotransition bijections
on bumpless pipe dreams, together
with similar known results on
ordinary pipe dreams, one can establish
bijections between ordinary pipe dreams
and bumpless pipe dreams.
\begin{defn}
A (reduced) \textbf{bumpless pipe dream} is a tiling of the $n\times n$ grid
with the six kinds of tiles shown below

\begin{figure}[h]
\centering
\includegraphics[scale=0.5]{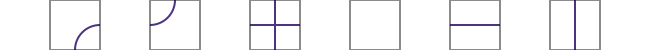}
\end{figure}

such that
\begin{enumerate}[(a)]
    \item there are $n$ pipes total,
    \item travelling from south to east, 
    each pipe starts vertically at the south edge of the grid, and
    ends horizontally at the east edge of the grid, and
    \item no two pipes cross twice.
\end{enumerate}
\end{defn}
Condition (c) is the reducedness condition. In
this paper we only consider reduced bumpless pipe dreams.
For convenience, we call these tiles \textbf{
r-tile}, \textbf{j-tile}, \textbf{``$+$''-tile}, \textbf{blank tile,}
\textbf{$\quo{-}$-tile}, and \textbf{$\quo{|}$-tile}.
The term ``bumpless'' comes from the fact that the tiling
disallows the ``\textbf{bump tile}'' shown below.

\begin{figure}[h]
\centering
\includegraphics[scale=0.5]{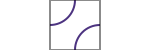}
\end{figure}

We index the tiles in a bumpless
pipe dream with \emph{matrix coordinates}.
Given a bumpless pipe dream, one can read off a permutation by
labeling the pipes from 1 to $n$ along the south edge, follow the pipes
from south to east, and read the labels top-down along the east edge. 
Given a permutation $\pi\in S_n$, we denote the set of bumpless pipe dreams
associated to $\pi$ by $\BPD(\pi)$.

\begin{figure}[h]
\centering
\includegraphics[trim=45 1250 500 45,clip, scale=0.09]{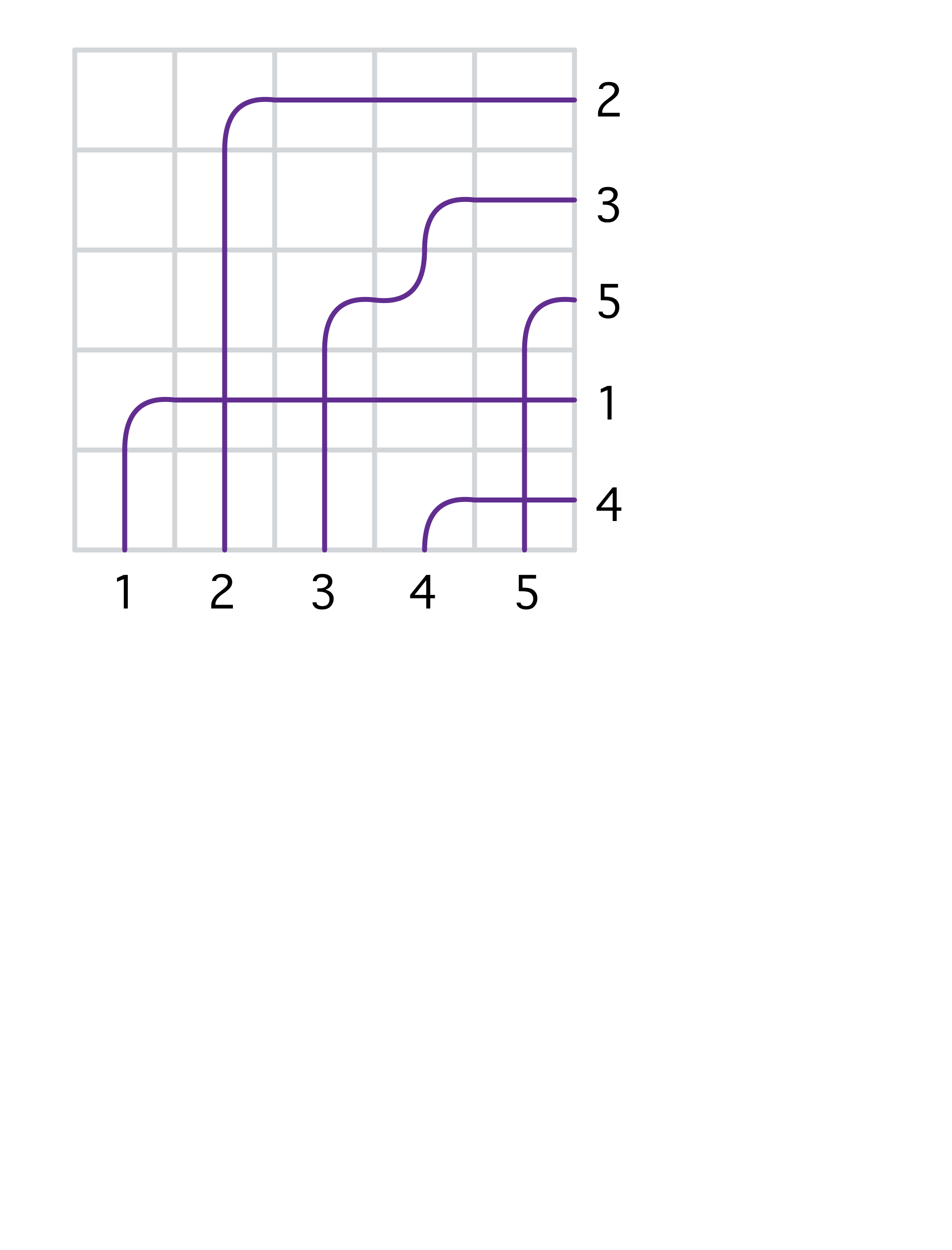}
\caption{A bumpless pipe dream for $\pi=23514$}
\label{fig:23514}
\end{figure}

For example, in Figure \ref{fig:23514}, the j-tile
at $(3,4)$ belongs to pipe
$\pi(2)=3$.

\begin{defn}
For $D\in \BPD(\pi)$, let $blank(D)\subseteq[n]\times [n]$ denote the 
set of blank tiles in the bumpless pipe dream $D$.
For $\pi\in S_n$, let 
\[\S_\pi (\mathbf{x,-y}):= \sum_{D\in \BPD(\pi)}\prod_{(i,j)\in blank(D)}(x_i-y_j)\in
\Z[x_1,\cdots,x_n,y_1,\cdots, y_n].\]
\end{defn}
Lam, Lee, and Shimozono showed that
that $\mathfrak{S}_\pi(\mathbf{x,-y})$ is the 
double Schubert polynomial for $\pi\in S_n$
\cite[Theorem 5.13]{lam2018back}.
Setting all the $y$ variables to 0, we get 
the expression for single Schubert polynomials
\[\S_\pi (\mathbf{x}):= \sum_{D\in \BPD(\pi)}\prod_{(i,j)\in blank(D)}x_i\in \Z[x_1,\cdots,x_n].\]
\ssection{A Bijective Proof of Monk's Rule with Bumpless Pipe Dreams}
\begin{thm}[Monk's rule] Let $\pi\in S_n$, $1\le \alpha <n$,
such that there exists some $l>\alpha$ such that
$\pi\,t_{\alpha,l}\gtrdot \pi$, where 
$\gtrdot$ denotes the covering relation in Bruhat order,
and $t_{a,b} $ denotes transposition of $a$ and $b$
in $S_n$.
 Then
\[\S_\alpha (\mathbf{x})\S_\pi(\mathbf{x}) = \sum_{\substack{k\le\alpha<l \\ \pi\,t_{k,l}\gtrdot\pi}}\S_{\pi\,t_{k,l}}(\mathbf{x}).\]
\begin{rmk}Note that we can remove the conditions on $\alpha$ and
the existence of $l$ such that
$\pi\,t_{\alpha,l}\gtrdot \pi$ if we consider
$\pi\in S_\infty := \bigcup_{n}S_n$, and this is how
the rule is usually stated. 
For convenience of our combinatorial proofs
we choose
to work with the version stated
for $\pi\in S_n$, but this does
not lose the level of generality.
\end{rmk}
Subtracting $\S_{\alpha-1}(\mathbf{x})\S_\pi(\mathbf{x})$ from it and rearranging,
we get
\begin{equation}
\label{singlemonk}
x_\alpha\S_\pi(\mathbf{x}) + \sum_{\substack{k<\alpha\\\pi\,t_{k,\alpha}\gtrdot \pi}}\S_{\pi\,t_{k,\alpha}}(\mathbf{x}) =\sum_{\substack{\alpha<l\\\pi\,t_{\alpha,l}\gtrdot \pi}} \S_{\pi\,t_{\alpha,l}}(\mathbf{x})
\end{equation}
\end{thm}
The goal of this section is to give a bijective proof
of formula (\ref{singlemonk}) with bumpless pipe dreams, as stated
in the following theorem.
\begin{thm}
\label{thm:monkbpd}
Given $\pi\in S_n$ and
$1\le \alpha <n$ such that there exists
$l>\alpha$ where $\pi\,t_{\alpha,l}\gtrdot \pi$, there 
exists a bijection
 \[\Phi_\pi:\BPD(\pi)\sqcup\coprod_{\substack{k<\alpha\\ \pi\,t_{k,\alpha}\gtrdot \pi}}\BPD(\pi\,t_{k,\alpha}) \longrightarrow
\coprod_{\substack{\alpha<l\\ \pi\,t_{l,\alpha}\gtrdot \pi}}\BPD(\pi\,t_{l,\alpha}).\]
such that for any $D\in \BPD(\pi)$, the number of blank
tiles on each row other than row $\alpha$ is preserved
under the map, the number of blank tiles on row
$\alpha$ increases by 1, and for any 
$D\in \bigcup_{\substack{k<\alpha\\ \pi\,t_{k,\alpha}\gtrdot \pi}}\BPD(\pi\,t_{k,\alpha})$ the number of blank tiles
on each row is preserved under the map.
\end{thm}
We start by preparing a few technical lemmas.
In  \cite[Section 5.2]{lam2018back}, the authors defined \textbf{droop moves} on bumpless
pipe dreams. We use the same language here.
Define an \textbf{almost bumpless pipe dream}
of $\pi$ at $(i,j)$ by allowing 
a bumpless pipe dream diagram to have
exactly one bump
tile at position $(i,j)$. (Double crossing of two
pipes is still not allowed.)  Note that an almost bumpless
pipe dream may be created from a bumpless pipe dream
by drooping a pipe into
an r-tile (or undrooping into a j-tile), or replacing a ``$+$''-tile with a bump
tile without creating a double crossing. We also introduce
the terminology \textbf{r/j-shaped
turn} to refer to the corresponding
pipe segments in an r/j-tile or
bump tile.

\begin{lemma}
\label{lem:droop}
Let $(i,j)$ be the position of an r-shaped turn
of pipe $p=\pi(x)$. If there exists
$y>x$ such that $\pi\,t_{x,y}\gtrdot \pi$,
then there exist $a,b>0$ such that
$(i,j+b)$ and $(i+a,j)$ are not ``$+$''-tiles.
Pick the smallest such possible $a, b$, then
$p$ is allowed to droop into
$(i+a, j+b)$ \emph{with the possibility of creating
a bump in $(i+a, j+b)$ (but not a double crossing)}.
\end{lemma}

\begin{figure}[H]
    \centering
    
    \subfloat{\includegraphics[trim=50 1400 500 50,clip, scale=0.1]{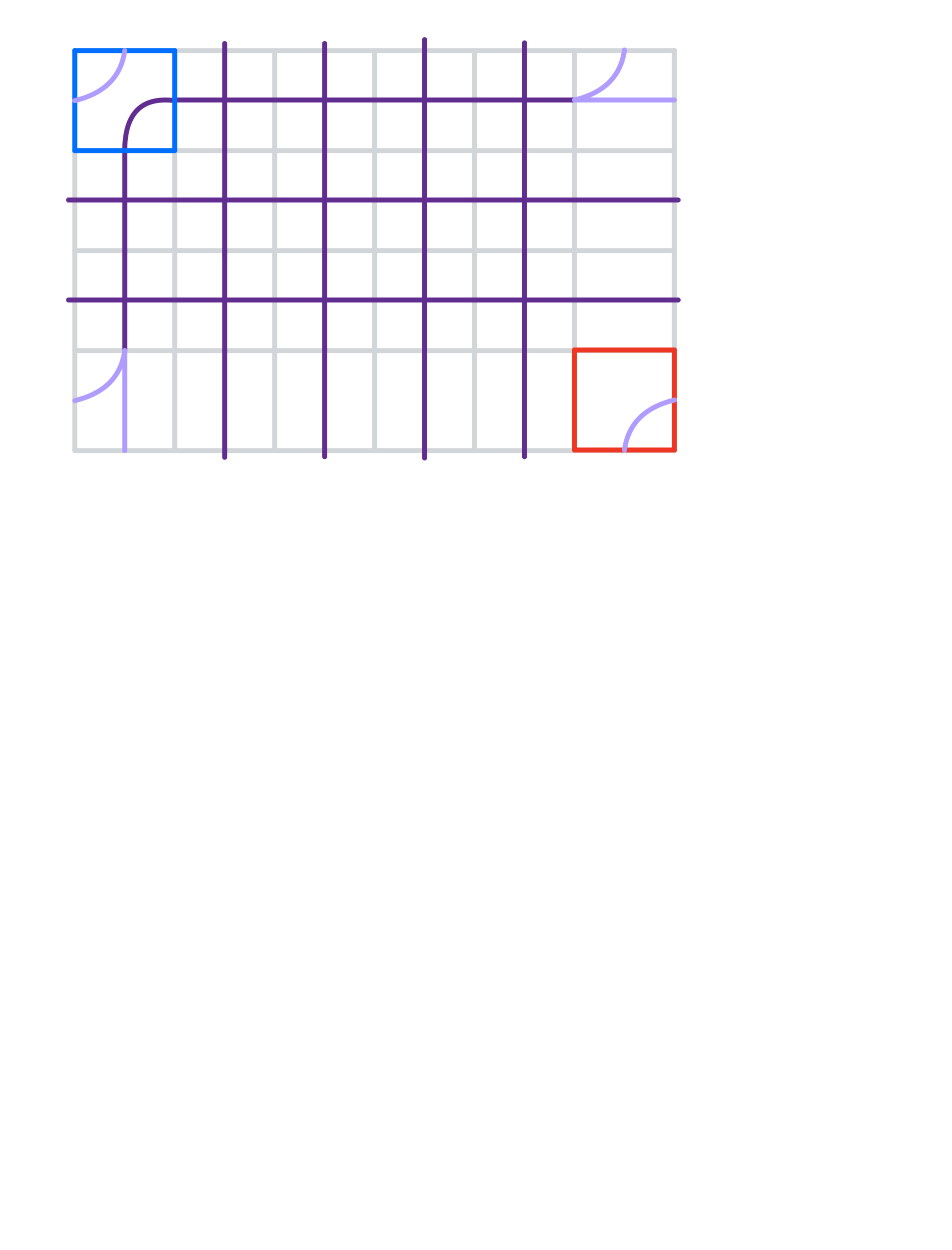} }%
    \qquad
    \subfloat{\includegraphics[trim=50 1400 500 50,clip, scale=0.1]{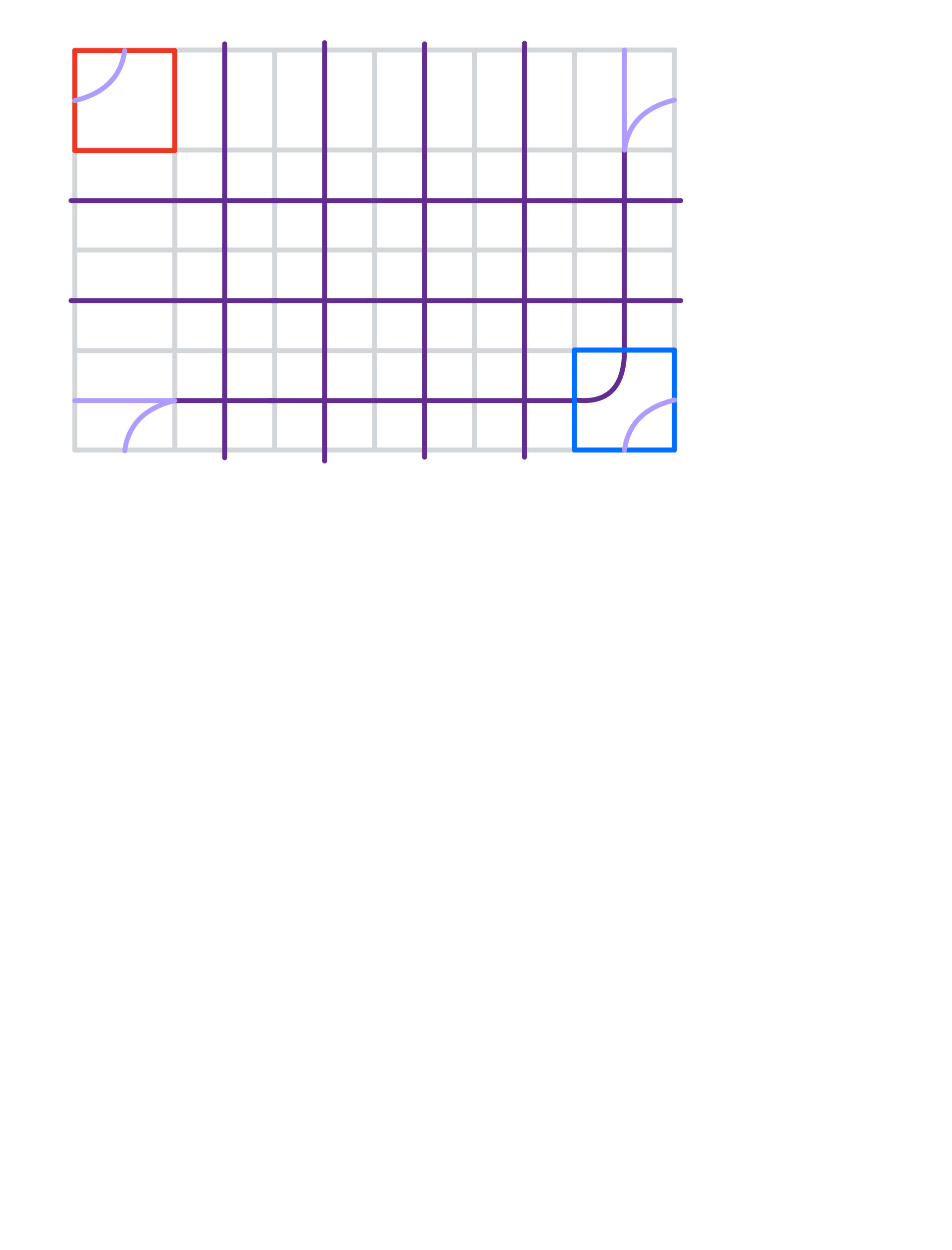} }%
    \caption{Droop to the closest tile (light purple indicates possibilities)}
    \label{fig:droop}
\end{figure}

\begin{proof}
Suppose for all $j'>j$, $(i,j')$ is a ``$+$''-tile.
Then since all pipes need to exit from the
east edge, the only way to fill the region
$(i',j')$ with $i'\ge i$, $j'\ge j$, 
$(i,j)\neq(i',j')$ is with ``$+$''-tiles. 
This implies that there is no $y>x$ such that
$\pi\,t_{x,y}\gtrdot \pi$. The same reasoning 
applies if for all $i'>i$, $(i',j)$ is a 
``$+$''-tile. 

Now pick the smallest $a,b$ as stated in the
lemma. Observe that in this case
$(i,j+b)$ is either a ``$-$''-tile or 
a j-tile,  $(i+a,j)$ is either a ``$|$''-tile
or a j-tile, and all tiles $(i',j')$
with $i<i'<i+a$ and $j<j'<j+b$ must be 
``$+$''-tiles. This means that
all $(i',j+b)$ for $i<i'<i+a$ must be ``$-$''-tiles, and all $(i+a,j')$ for
$j<j'<j+b$ must be ``$|$''-tiles. 
Therefore, the tile at $(i+a,j+b)$ has 
a ``$-$''-tile above and a ``$|$''-tile
to the left, so it can only be a blank or 
r-tile. It is then easy to see
$p$ may droop into $(i+a,j+b)$, with the 
possibility of creating a bump but not
a double crossing.
\end{proof}

Intuitively, Lemma \ref{lem:droop} is about
finding the closest tile
an r-shaped corner can droop
into.
It is not hard to see that the
droop move described in this lemma
has an inverse operation.

\begin{lemma}
\label{lem:undroop}
Let $(i,j)$ be the position of a j-shaped turn
of pipe $p=\pi(x)$. Pick the largest
$a,b >0$ such that the tiles on row
$i$ strictly
between $(i,j-b)$ and $(i,j)$ are ``$+$''-tiles
and the tiles on column $j$ strictly between
$(i-a,j)$ and $(i,j)$ are ``$+$''-tiles.
Then $p$ is allowed to undroop 
to $(i-a, j-b)$, with the possibility
of creating a bump in $(i-a,j-b)$.
\end{lemma}
\begin{proof}
Note that $a,b$ always exist since
bumpless pipe dreams cannot have crosses
on the north or west border. The rest of
the proof is symmetric to the proof of the
second half of
Lemma \ref{lem:droop}.
\end{proof}

\begin{lemma}
\label{lem:lowerr}
Suppose  $p=\pi(x)$ and $q=\pi(y)$ are two 
pipes that cross once
and bump once, and that 
the j-shaped corner in the bump tile belongs to 
$p$. If we swap the positions of 
the cross and the bump, then in the new bump tile,
the r-shaped turn belongs to $p$.
\end{lemma}

\begin{proof}
Suppose $p$ and $q$ cross at $(i,j)$ before the swap.
Consider the pipes travelling from 
south to east. 
If the bump is after the
cross, we must have $p=\pi(x)>q=\pi(y)$
and $x<y$, namely the ``$|$'' in the cross
at $(i,j)$ must belong to $p$. After the swap, 
$p$ still enters from the bottom of $(i,j)$, and
therefore it makes an r-shaped turn. If the bump
is before the cross, we must have 
$q=\pi(y)>p=\pi(x)$ and $y<x$, namely the
``$-$'' in the cross at $(i,j)$ must belong to $p$.
After the swap, $p$ still exits from the right,
and therefore makes an r-shaped turn at $(i,j)$.
\end{proof}
\begin{figure}[H]
    \centering
    \includegraphics[scale=0.1]{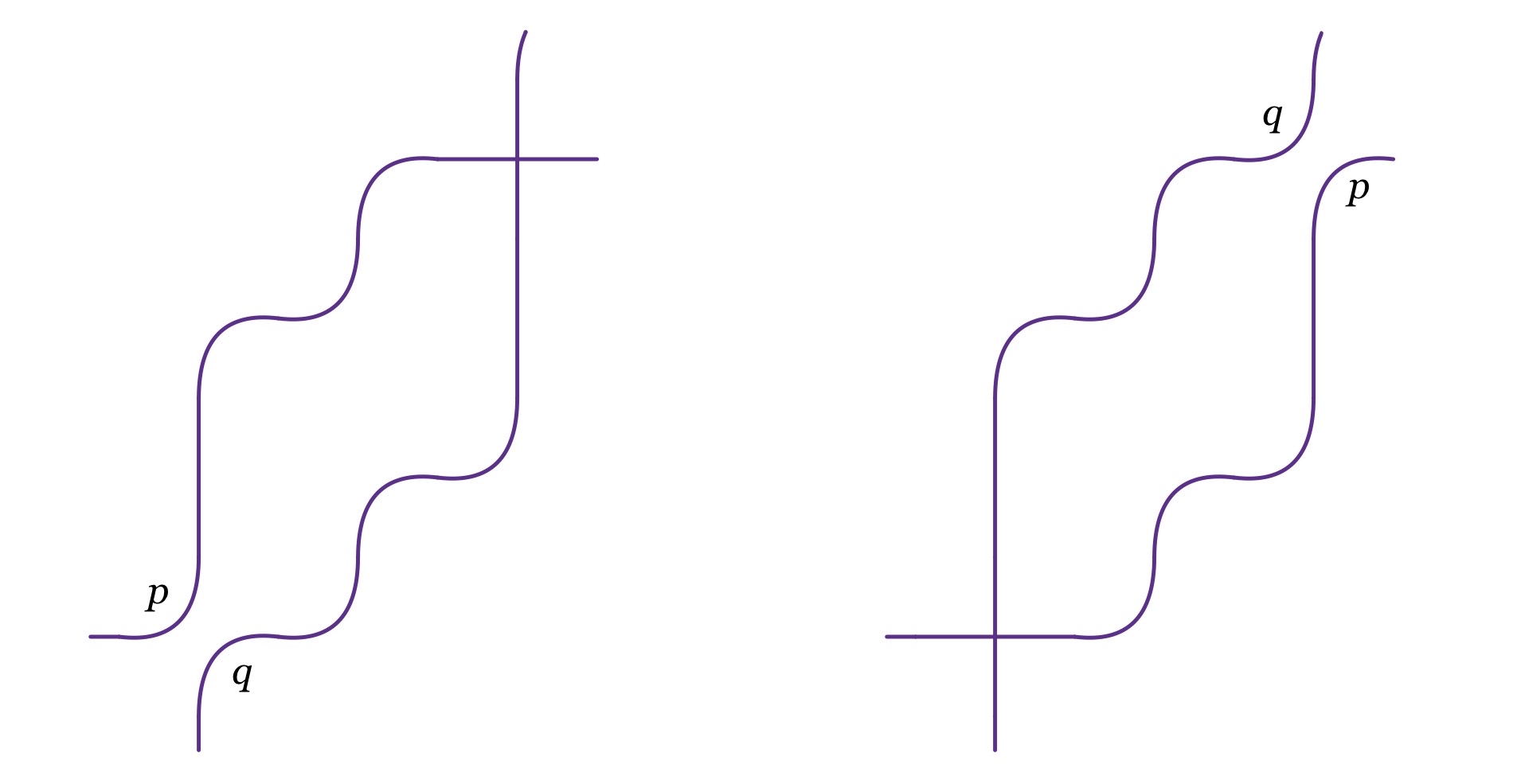}
    \caption{A case of bump-cross swap}
    \label{fig:crossbump}
\end{figure}
Again, we have the opposite version of
this statement, whose proof we omit.

\begin{lemma}
\label{lem:upperj}
Suppose  $p=\pi(x)$ and $q=\pi(y)$ are two 
pipes that cross once
and bump once, and that 
the r-shaped corner in the bump tile belongs to 
$p$. If we swap the positions of 
the cross and the bump, then in the new bump tile,
the j-shaped turn belongs to $p$.
\end{lemma} 

\begin{figure}[H]
    \centering
    \begin{tikzpicture}
    \node[anchor=south west,inner sep=0] (image) at (0,0)  
     {\includegraphics[trim=45 1350 500 45,clip, scale=0.1]{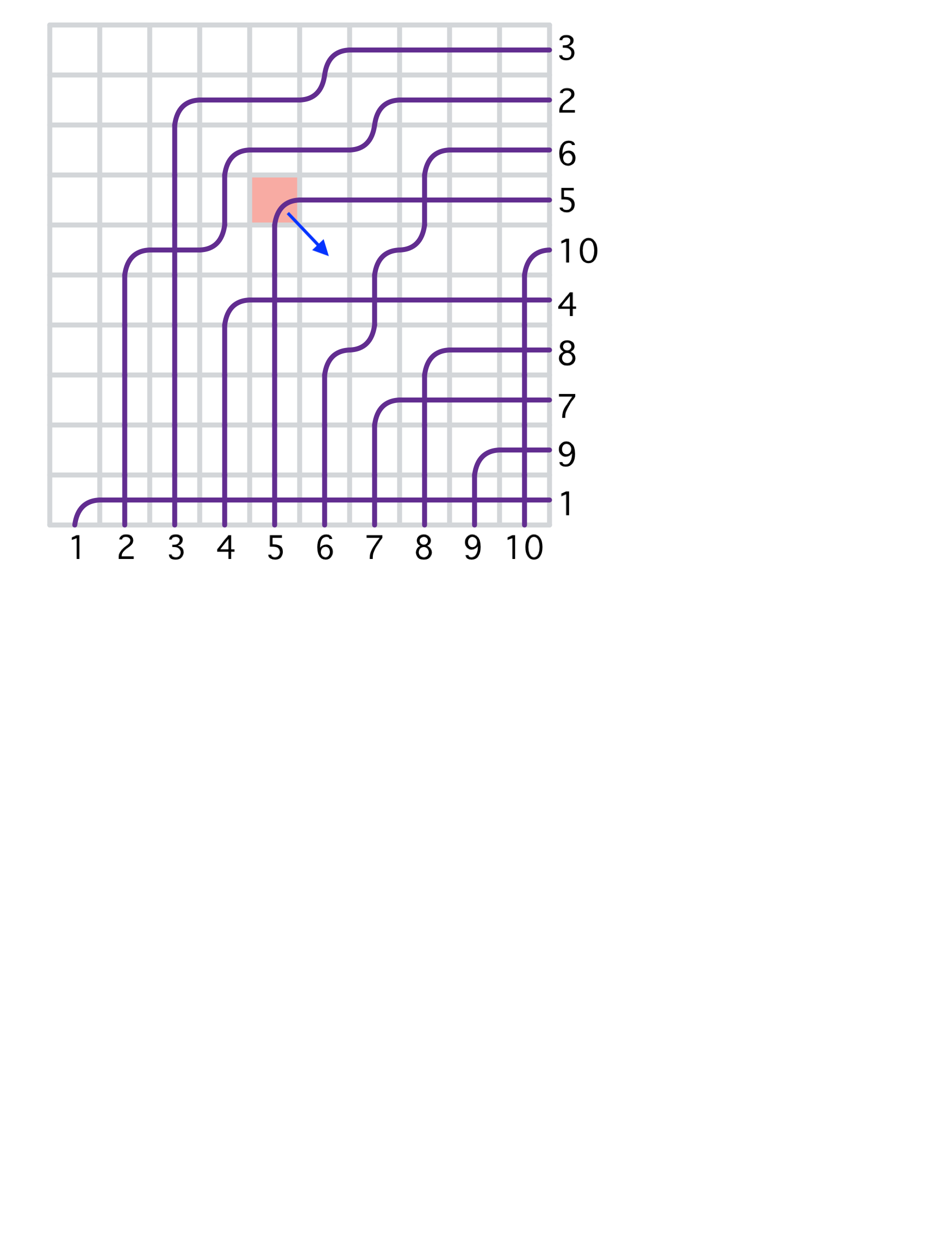}};
    \draw[black, thick, ->] (4.2,2) -- (4.8,2);
    \node[anchor=south west,inner sep=0] (image) at (5,0)  
     {\includegraphics[trim=45 1350 500 45,clip, scale=0.1]{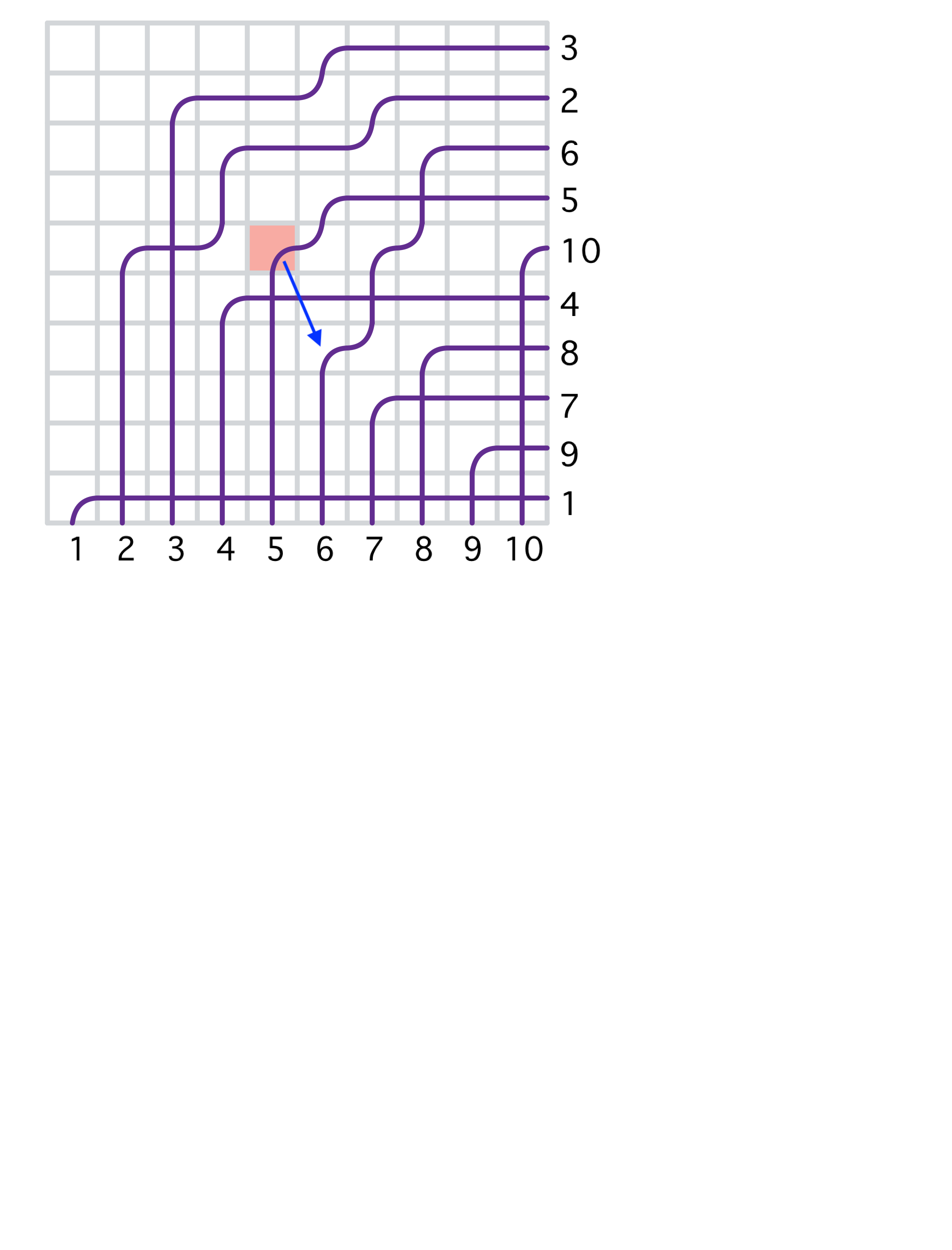}};
    \draw[black, thick, ->] (9.2,2) -- (9.8,2);
    \node[anchor=south west,inner sep=0] (image) at (10,0)
     {\includegraphics[trim=45 1350 500 45,clip, scale=0.1]{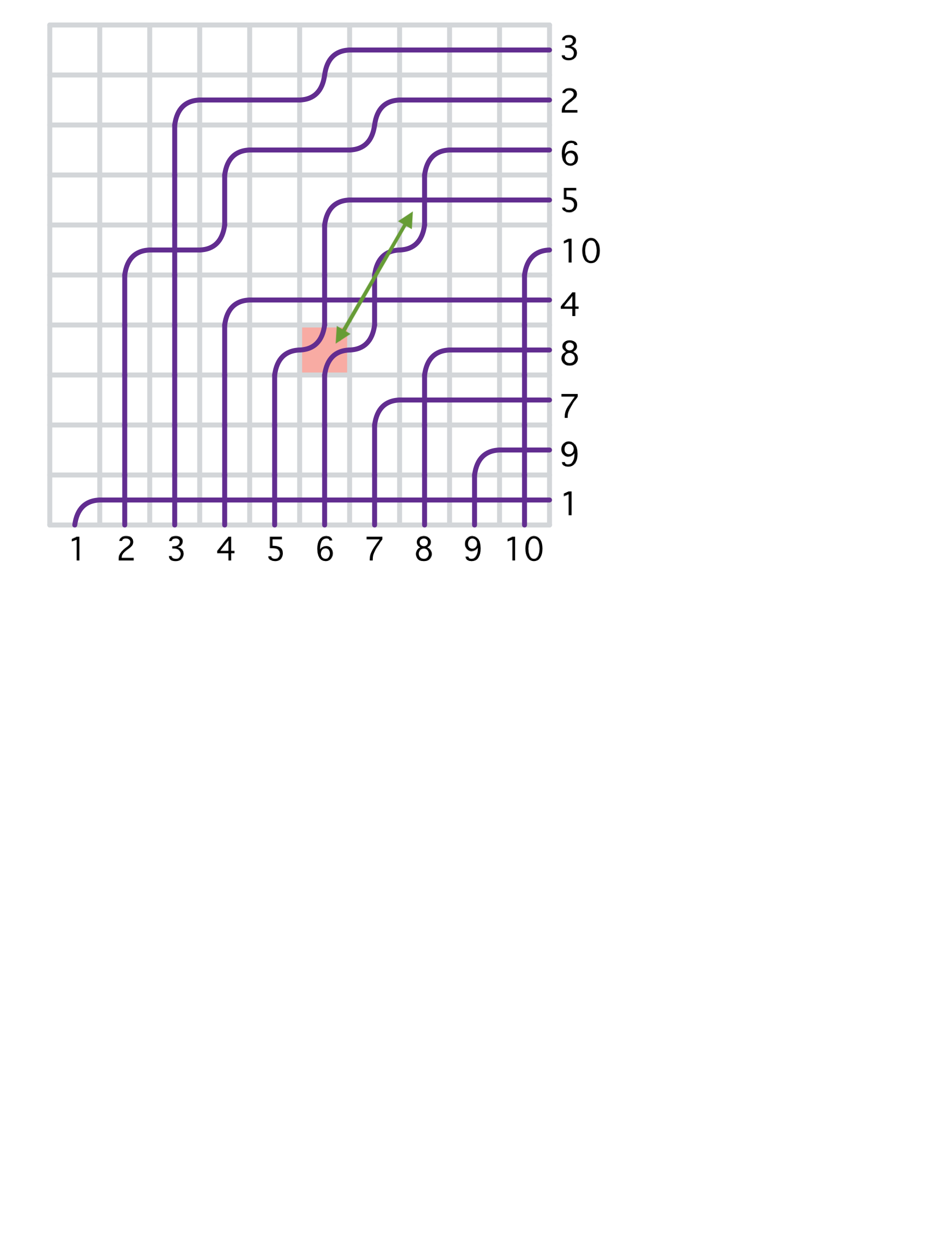}};
    \draw[black, thick, ->] (2,-3) -- (2.6,-3);
    \node[anchor=south west,inner sep=0] (image) at (2.8,-5)
     {\includegraphics[trim=45 1350 500 45,clip, scale=0.1]{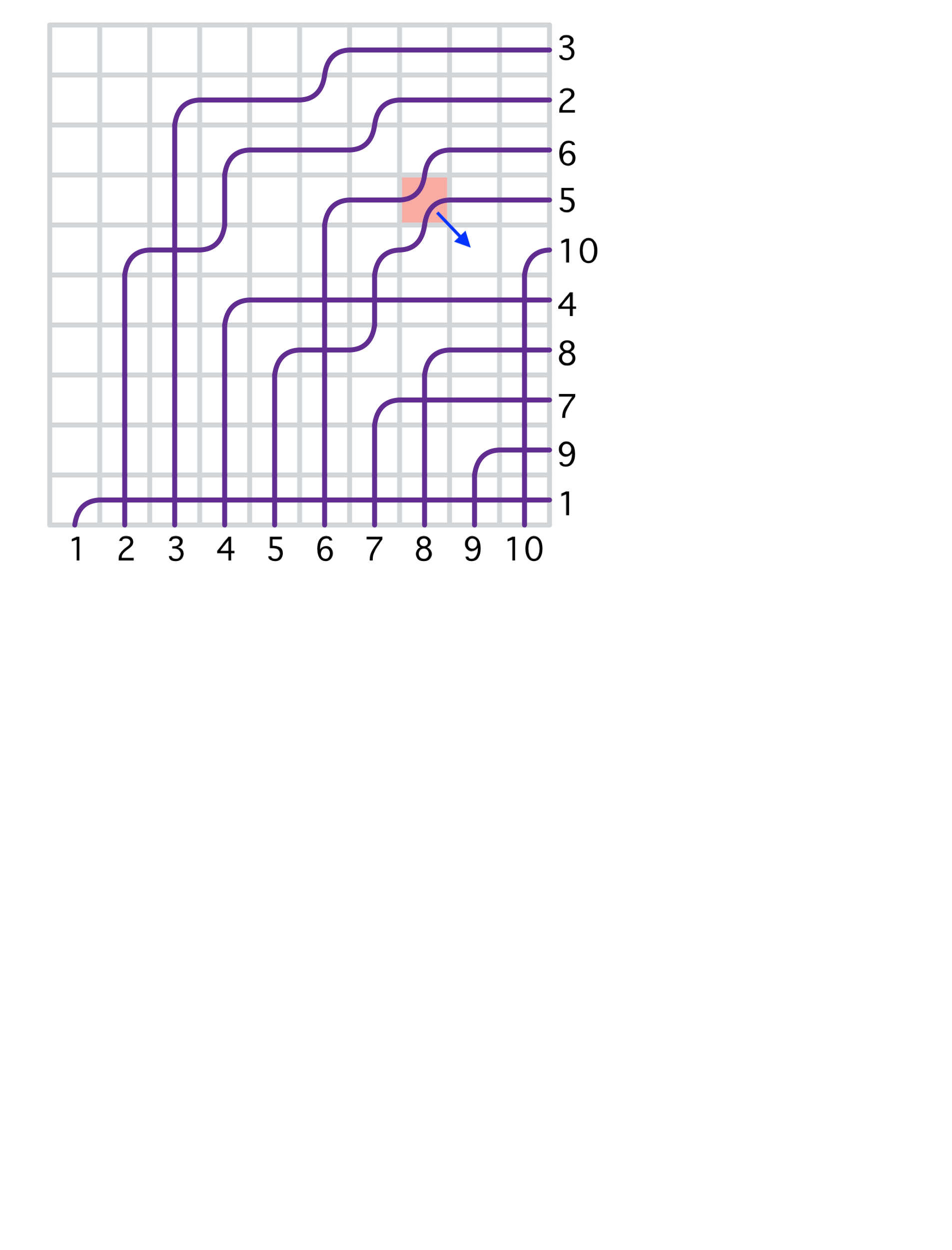}};
    \draw[black, thick, ->] (7,-3) -- (7.6,-3);
    \node[anchor=south west,inner sep=0] (image) at (7.8,-5)
     {\includegraphics[trim=45 1350 500 45,clip, scale=0.1]{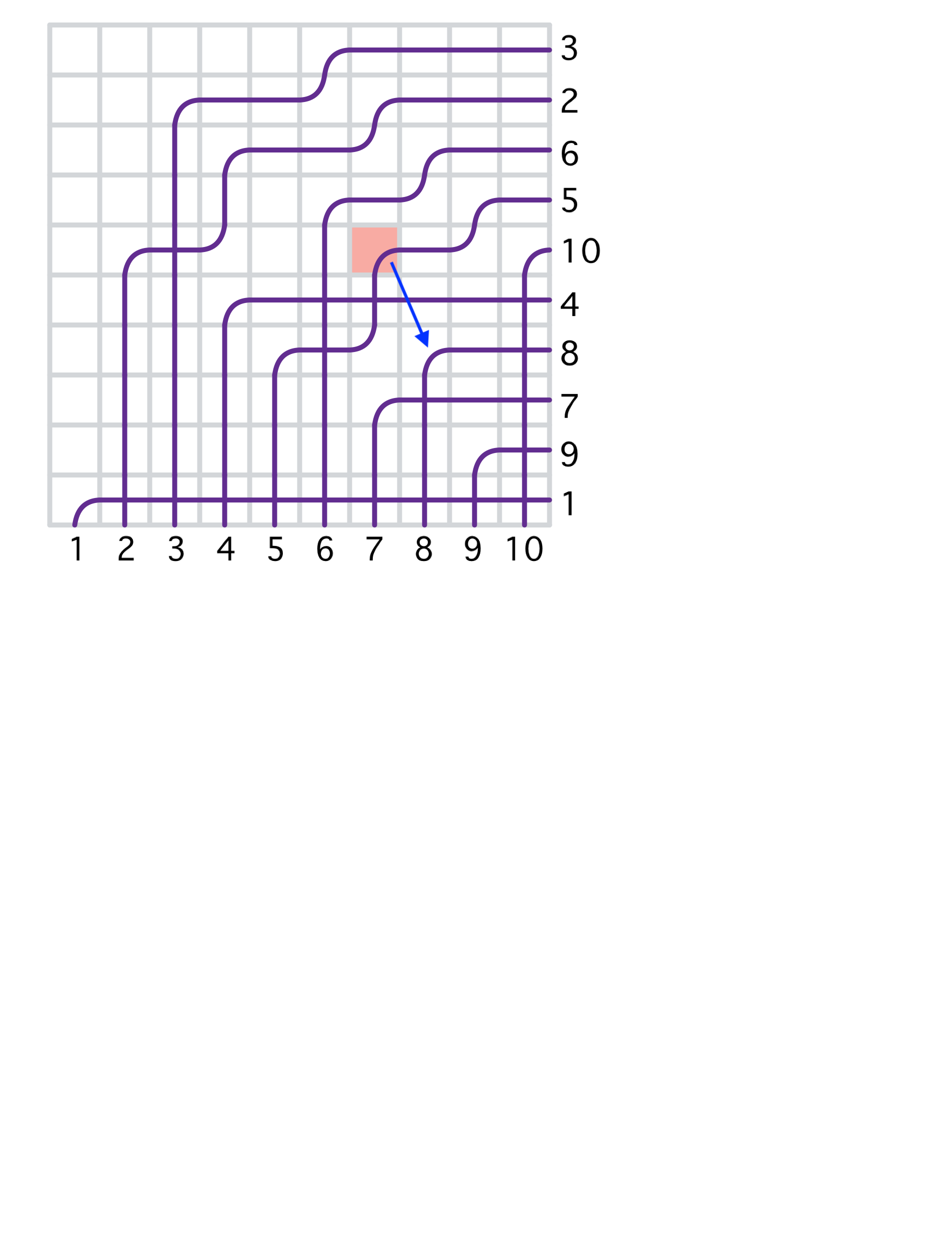}};
    \draw[black, thick, ->] (2,-8) -- (2.6,-8);
    \node[anchor=south west,inner sep=0] (image) at (2.8,-10)
     {\includegraphics[trim=45 1350 500 45,clip, scale=0.1]{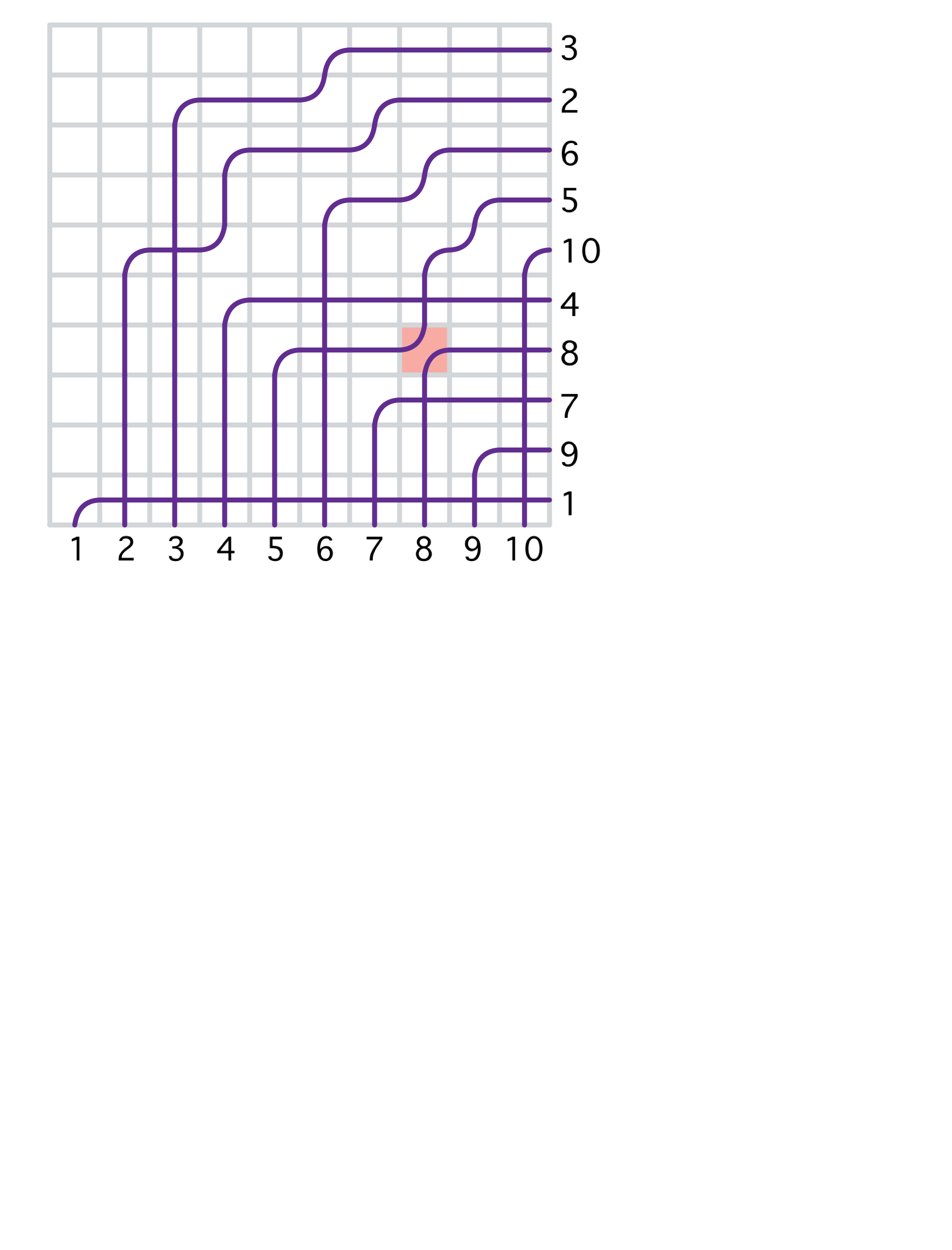}};
    \draw[black, thick, ->] (7,-8) -- (7.6,-8);
    \node[anchor=south west,inner sep=0] (image) at (7.8,-10)
     {\includegraphics[trim=45 1350 500 45,clip, scale=0.1]{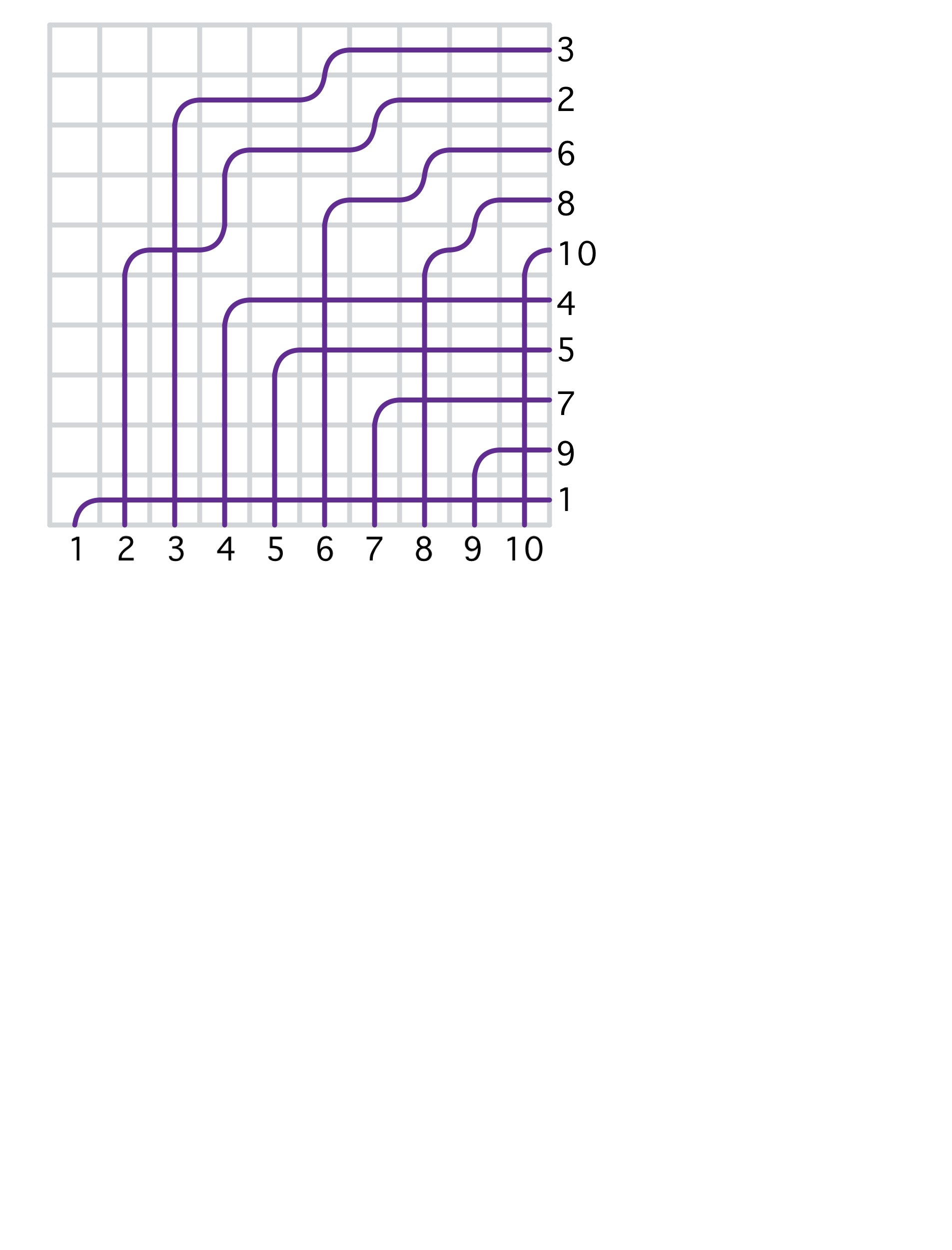}};
\end{tikzpicture}
\caption{Insertion of a blank tile at
      $(i,j)=(4,5)$ for a BPD of
      $\pi=[3,2,6,5,10,4,8,7,9,1]$}
    \label{fig:example}
\end{figure}
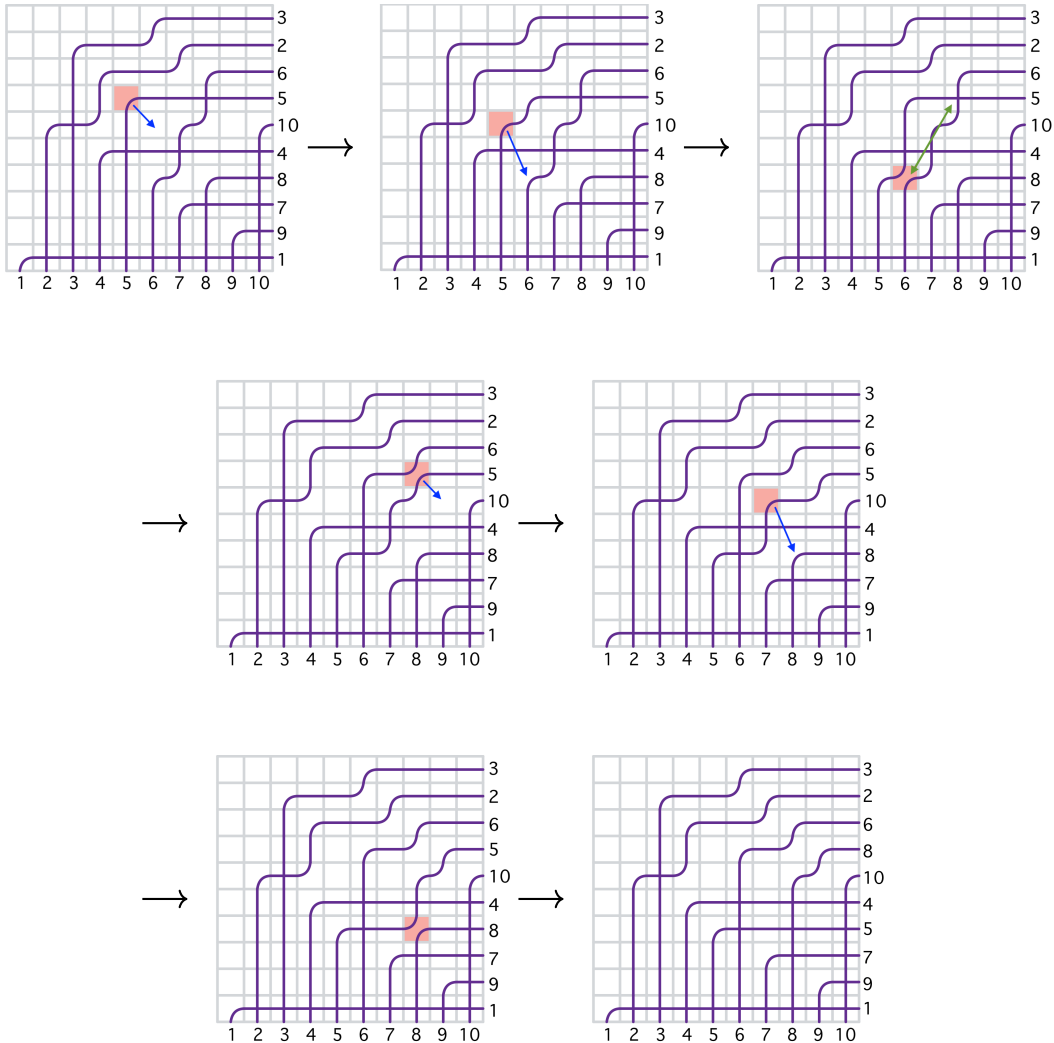

Figure \ref{fig:example} shows
a walk-through of the
algorithm below, where we start with inserting a blank
tile at an r-tile on row 4.
The reader is invited to 
guess the the algorithm
before reading the description.
The shaded square in each diagram 
denotes  the r-tile
at which a blank tile is about to be inserted,
or a bump tile that needs to be resolved.

We now describe an algorithm for inserting a blank tile
at position $(i,j)$ where there is an r-tile
or, as will be made clear below, to resolve a conflict
where there is temporarily a bump tile. Suppose this r-shaped corner belongs to pipe $p=\pi(x)$.
Let $(i+a,j+b)$,
$a,b>0$, be the tile southeast to $(i,j)$ such that
the tiles on the $i$th row strictly between $(i,j)$ and $(i, j+b)$ are all
``$+$''-tiles, and the tiles in the $j$th column between $(i,j)$
and $(i+a,j)$ are all ``$+$''-tiles. 
By Lemma \ref{lem:droop}, $p$ may droop
into $(i+a,j+b)$ with the possibility of creating
a bump.
We let the pipe $p$ droop into $(i+a,j+b)$. 
 If $(i+a, j+b)$ used to be a blank tile (now a ``j''),
we have newly occupied a blank tile on row $i+a$, so 
we find the r-tile on row $i+a$ that belongs to $p$ and 
repeat the same algorithm for inserting a blank tile at
an r-tile, as before. 
(Note that such an r-tile always exists.)
The other possibility is that
$(i+a, j+b)$ used to be an r-tile (now a bump). 
Suppose the  bump is with pipe $q=\pi(y)$. If
$\pi\,t_{x,y}\gtrdot\pi $, or in other words, $p$ and $q$ 
do not cross, we replace the bump tile
with a ``$+$''-tile, and terminate the algorithm. If $p$
and $q$  cross each other, we find the tile $(i',j')$
where the crossing is, replace the existing bump tile in 
$(i+a, j+b)$ with a cross, and replace the cross in $(i',j')$
with a bump tile. After this, by  
Lemma \ref{lem:lowerr},
the r-shaped turn
in $(i',j')$ must belong to $p$. We resolve this bump
by going back to the beginning of the algorithm and repeat the
process.

We give the pseudocode of the algorithm below. 
Let $\BPD'_{i,j}(\pi)$ denote the set of bumpless pipe
dreams of $\pi$ that have an r-tile at $(i,j)$, plus the
\emph{almost} bumpless pipe dreams of $\pi$ 
which have exactly one
bump at position $(i,j)$.
For $D\in \BPD'_{i,j}(\pi)$, let $D(m,n)$ denote
the tile in $D$ at position $(m,n)$.
\vskip 0.3em
{\small
\begin{algorithm}[H]
\SetKwProg{Fn}{}{ } {}\SetKwFunction{Fun}{insert\_blank\_or\_resolve\_bump\_at\_r}%
\Fn(){\Fun{D,i,j}}{
\KwData{$D\in \BPD'_{i,j}(\pi)$, where the r-shaped turn
in $D(i,j)$ belongs to pipe $p=\pi(x)$, which satisfies
$\exists y>x$ such that $\pi\,t_{x,y}\gtrdot \pi$.}
$a,b\gets 1,1$\;
\lWhile{$D(i+a,j)=\quo{+}$}{$a\gets a+1$}
\lWhile{$D(i,j+b)=\quo{+}$}{$b\gets b+1$}
Droop $p$ into $(i+a, j+b)$\;
\uIf{$D(i+a,j+b)=\quo{\text{j}}$}
    {$(i+a, j') \gets$ position of r-tile of $p$ on row $i+a$\;
     \Fun($D,i+a,j'$)}
    \uElse{Let $q=\pi(y)$ be the pipe that $p$ bumps  into at $(i+a,j+b)$\;
     \uIf{$\pi\,t_{x,y}\gtrdot \pi$}
         {$D(i+a,j+b)\gets\quo{+}$\;
          \Return{$D$} }
     \uElse{$(i',j')\gets$ position of existing cross of $p$ and $q$\;            
            $D(i',j') \gets$ bump tile\;
            $D(i+a,j+b)\gets\quo{+}$\;
            \Fun($D,i',j'$)}}

}

\caption{}
\label{alg:forward}
\end{algorithm}
}
\begin{prop}
\label{prop:terminate}
The algorithm  \funfont{insert\_blank\_or\_resolve\_bump\_at\_r} terminates
and produces a bumpless pipe dream of $\pi\,t_{x,y}$
for some $y>x$ such that $\pi\,t_{x,y}\gtrdot \pi$.
\end{prop}
\begin{proof}
The well-definedness of the algorithm follows from Lemmas
\ref{lem:droop} and \ref{lem:lowerr}, as explained in the
construction. For termination,
observe that we modify the pipe $p$ 
either by droop moves or  cross-bump-swap moves. The area under the pipe $p$ (as a curve) 
in the $n\times n$ square
strictly decreases after each of these moves. Since the modification to the
diagram in each iteration of the function 
before the final modification at line 12 right before
it returns preserves
the property that the diagram is a
BPD or an almost BPD of $\pi$, there is a finite set of
possible areas under the pipe $p$. Therefore, the
algorithm must terminate, and by the terminating condition, 
$p$ must have bumped into $q$ after drooping,
therefore occupying the j-shaped corner at this bump.
This means $p<q$, so $x<y$. 
Therefore,
it produces a bumpless pipe dream of $\pi\,t_{x,y}$
for some $y>x$ such that $\pi\,t_{x,y}\gtrdot \pi$.
\end{proof}

Now, the algorithm has an opposite version
which inserts a blank tile at a position
where there is a j-tile, or resolves
a conflict where there is a bump tile, in
the opposite direction. 
\vskip 0.5em
{\small
\begin{algorithm}[H]

\SetKwProg{Fn}{}{ } {}\SetKwFunction{Fun}{insert\_blank\_or\_resolve\_bump\_at\_j}%
\Fn(){\Fun{D,i,j}}{
\KwData{$D\in \BPD'_{i,j}(\pi)$, where the j-shaped turn
in $D(i,j)$ belongs to pipe $p=\pi(x)$}
$a,b\gets 1,1$\;
\lWhile{$D(i-a,j)=\quo{+}$}{$a\gets a+1$}
\lWhile{$D(i,j-b)=\quo{+}$}{$b\gets b+1$}
Undroop $p$ into $(i-a, j-b)$\;
\uIf{$D(i-a,j-b)=\quo{\text{r}}$}
    {\uIf{$\forall j'>j-b$, $D(i-a,j')$ does not have a j-tile}
         {\Return{$D$}}
     $(i-a, j') \gets$ position of j-tile of $p$ on row $i-a$\;
     \Fun($D,i-a,j'$)}
    \uElse{Let $q=\pi(y)$ be the pipe that $p$ bumps  into at $(i-a,j-b)$\;
     \uIf{$\pi\,t_{y,x}\gtrdot \pi$}
         {$D(i-a,j-b)\gets\quo{+}$\;
          \Return{$D$} }
     \uElse{$(i',j')\gets$ position of existing cross of $p$ and $q$\;            
            $D(i',j') \gets$ bump tile\;
            $D(i+a,j+b)\gets\quo{+}$\;
            \Fun($D,i',j'$)}}

}
\caption{}
\label{alg:backward}
\end{algorithm} }

\begin{prop}
\label{prop:terminate2}
The algorithm \funfont{insert\_blank\_or\_resolve\_bump\_at\_j}
terminates and produces a either a bumpless pipe dream 
of $\pi\, t_{y,x}$ for some $y<x$, or a bumpless pipe dream
of $\pi$ with one fewer blank tile on row $x$, compared to the input. 
\end{prop}
\begin{proof}
By Lemmas \ref{lem:undroop} and
\ref{lem:upperj}, this algorithm is
well-defined. By similar reasoning as in
Proposition \ref{prop:terminate}, this 
algorithm terminates. 

If it terminates
by triggering the condition on line 7, $p$ only turns 
once on row $i-a$, so this must also be the row
in which $p$ exits, which means $i-a=x$. This entire process
does not change the permutation, so the output is a bumpless
pipe dream of $\pi$. The last undrooping step ate a
blank tile on row $x$ and did not give it back, so this
bumpless pipe dream has one fewer blank tile on row $x$.

If the algorithm terminates by triggering the condition
on line 13, $p$ must have bumped into $q$ after undrooping,
therefore occupying the r-shaped corner at this bump.
This means $p>q$, and therefore $x>y$. 
\end{proof}

We are now ready to describe the bijection $\Phi_\pi$.

\noindent\emph{Proof of Theorem \ref{thm:monkbpd}}.\ \  Let $D\in \BPD(\pi)$, $p=\pi(\alpha)$.
Pipe $p$ exits on row $\alpha$. Let 
$(\alpha, j)$ be the position of the r-tile of $p$
on row $\alpha$. We run the function
\funfont{insert\_blank\_or\_resolve\_bump\_at\_r}
on $(D,\alpha,j)$. By Proposition \ref{prop:terminate},
the output is a bumpless pipe dream 
$D'\in \BPD(\pi\,t_{\alpha, l})$
for some $l>\alpha$ and $\pi\,t_{\alpha, l}\gtrdot \pi$. By construction of the algorithm,
the number of blank tiles on each row stays the
same, except for row $\alpha$ where $D'$ has one
more blank tile than $D$.

Let $D\in \BPD(\pi\,t_{k,\alpha})$ for some $k<\alpha$ such that
$\pi\, t_{k,\alpha}\gtrdot \pi$. Let $q=\pi\,t_{k,\alpha}(\alpha)$ and
$p=\pi\,t_{k,\alpha}(k)$ be pipes. Since $\pi \,t_{k,\alpha}\gtrdot \pi$, $p$ and $q$ must cross. Let $(i,j)$ be the position
of the tile where they cross. 
Notice that since $q<p$, the ``$|$'' segment in this cross
must belong to $p$. 
We now replace this ``$+$''
tile with a bump tile. Notice that by doing so we have uncrossed
$p$ and $q$, therefore creating an almost bumpless pipe dream,
$D'\in \BPD'_{i,j}(\pi)$. Also, now $p=\pi(\alpha)$, $q=\pi(k)$, and the r-shaped corner in this new bump tile
belongs to  pipe $p$.

We run the function
\funfont{insert\_blank\_or\_resolve\_bump\_at\_r}
on $(D',i,j)$. 
Again by Proposition \ref{prop:terminate},
the output is a bumpless pipe dream 
$D''\in \BPD(\pi\,t_{\alpha, l})$
for some $l>\alpha$ and $\pi\,t_{\alpha, l}\gtrdot \pi$. The number of blank tiles on each row remains
constant during this process.

To go the opposite direction, let 
$E\in \BPD(\pi\,t_{l,\alpha})$ for some $\alpha < l$ such that
$\pi\, t_{\alpha,l}\gtrdot \pi$.
Let $q=\pi\,t_{\alpha,l}(\alpha)$ and
$p=\pi\,t_{\alpha,l}(l)$ be pipes. Since $\pi \,t_{\alpha,l}\gtrdot \pi$, $p$ and $q$ must cross. Let $(i,j)$ be the position
of the tile where they cross. 
Notice that since $q>p$, the ``$|$'' segment in this cross
must belong to $q$. 
We now replace this ``$+$''
tile with a bump tile. Again this uncrosses $p$ and $q$,
creating an almost bumpless pipe dream
$E'\in \BPD'_{i,j}(\pi)$, and
making $p=\pi(\alpha)$ and $q=\pi(l)$. The j-shaped corner
in this new bump tile belongs to $p$. 

We run the function
\funfont{insert\_blank\_or\_resolve\_bump\_at\_j} 
on $(E', i,j)$. By Proposition \ref{prop:terminate2},
there are two possible outcomes. The output is either
some $E''\in \BPD(\pi\, t_{k,\alpha})$ for some
$k<\alpha$, in which case the number of blank tiles
on each row stays invariant,
or some $E''\in \BPD(\pi)$ that occupies one more blank
tile on row $\alpha$ as compared to $E'$. 

By the construction of the two algorithms, it is easy to see
that the processes described above are inverses of each other,
giving a bijection between the two sets of bumpless pipe dreams.\hfill \qedsymbol

\ssection{Monk's Rule for Double Schubert Polynomials}
The version of Monk's rule for double Schubert polynomials 
states that 
\begin{thm}[Monk's rule for double Schuberts]
Let $\pi\in S_n$, $1\le \alpha <n$,
such that there exists some $l>\alpha$ with
$\pi\,t_{\alpha,l}\gtrdot \pi$.
Then
\[\S_\alpha(\mathbf{x,-y}) \S_{\pi}(\mathbf{x,-y})=\sum_{\substack{k\le\alpha<l\\\pi\, t_{k,l}\gtrdot\pi}}\S_{\pi\,t_{k,l}}(\mathbf{x,-y}) +\sum_{i=1}^\alpha(y_{\pi(i)}-y_i)\S_\pi(\mathbf{x,-y}).\]
\end{thm}
Computing $\S_\alpha(\mathbf{x,-y})\S_\pi(\mathbf{x,-y})-\S_{\alpha-1}(\mathbf{x,-y})\S_\pi(\mathbf{x,-y})$ and
rearranging terms, we find 
\begin{equation}
\label{doublemonk}
    (x_\alpha-y_{\pi(\alpha)})\S_\pi (\mathbf{x,-y})+ \sum_{\substack{k<\alpha\\\pi\,t_{k,\alpha}\gtrdot \pi}}\S_{\pi\,t_{k,\alpha}}(\mathbf{x,-y}) =\sum_{\substack{\alpha<l\\\pi\,t_{\alpha,l}\gtrdot \pi}} \S_{\pi\,t_{\alpha,l}}(\mathbf{x,-y}).
\end{equation}
We give a bijective proof of formula (\ref{doublemonk})
in this section.

We will first need to introduce 
decorations on blank tiles of bumpless pipe dreams. 
A \textbf{decorated bumpless pipe dream}
of $\pi$ is a bumpless
pipe dream together with a decoration on the blank
tiles, where each blank tile must be decorated
with either an $\mathsf{x}$ or a $\mathsf{-y}$ label. 
Let $ \widetilde{\BPD}(\pi)$ be the set of
decorated bumpless pipe dreams of $\pi$.
In other words, \[\widetilde{\BPD}(\pi)=\{(D,f):D\in \BPD(\pi), f:blank(D)\to\{\mathsf{x,-y}\}\}.\]
Note that $| \widetilde{\BPD}(\pi)|=|\BPD(\pi)|\times 2^{|blank(D)|}$ for any $D\in\BPD(\pi)$. 
Expand the double Schubert polynomial as 
a sum of monomials, we get the following
expression
\[\S_\pi(\mathbf{x,-y})=\sum_{(D,f)\in\widetilde{\BPD}(\pi)}\mon(D,f),\]
where \[ \mon(D,f)=\prod_{\substack{(i,j)\in blank(D)\\f(i,j)=\mathsf{x}}} x_i\prod_{\substack{(i,j)\in blank(D)\\f(i,j)=\mathsf{-y}}} (-y_j).\]
Similarly, we define 
\[\widetilde{\BPD}_{i,j}'(\pi):=\{(D,f):
D\in\BPD'_{i,j}(\pi),f:blank(D)\to\{\mathsf{x,-y}\}\}.\]

The combinatorial version 
of formula (\ref{doublemonk}) is
stated as follows. 
\begin{thm}
Let $\pi\in S_n$, $1\le \alpha <n$,
such that there exists some $l>\alpha$ with
$\pi\,t_{\alpha,l}\gtrdot \pi$.
Then there exists a bijection
 \[\widetilde{\Phi}_\pi:(\{\mathsf{x,-y}\}\times\widetilde{\BPD}(\pi))\sqcup\coprod_{\substack{k<\alpha\\ \pi\,t_{k,\alpha}\gtrdot \pi}}\widetilde{\BPD}(\pi\,t_{k,\alpha}) \longrightarrow
\coprod_{\substack{\alpha<l\\ \pi\,t_{l,\alpha}\gtrdot \pi}}\widetilde{\BPD}(\pi\,t_{l,\alpha}),\]
such that for any 
$(D,f)\in \widetilde{\BPD}(\pi))$,
\[\mon(\widetilde{\Phi}_\pi(\mathsf{x}, D,f)) = x_\alpha\mon(D,f),\] 
\[\mon(\widetilde{\Phi}_\pi(\mathsf{-y}, D,f)) = -y_{\pi(\alpha)}\mon(D,f),\]
and for any $k<\alpha, \pi\,t_{k,\alpha}\gtrdot \pi$,
$(D,f)\in \widetilde{\BPD}(\pi\,t_{k,\alpha})$,
\[\mon(\widetilde{\Phi}_\pi( D,f)) = \mon(D,f).\] 
\end{thm}
\begin{proof}

We will modify the algorithms given in the previous section
slightly to get the bijection
$\widetilde{\Phi}_\pi$.
The algorithms in both directions will now
take as input $(D,f)\in \widetilde{\BPD}'_{i,j}(\pi)$, \emph{as
well as a label $u\in\{\mathsf{x,-y}\}$ in the
case when
$D(i,j)$ is an r-tile in the
forward direction, and when 
$D(i,j)$ is a j-tile 
in the backward direction}. The outputs will also be 
decorated bumpless pipe dreams,
as well as a label $v\in\{\mathsf{x,-y}\}$ in 
certain cases.

In Algorithm \ref{alg:forward}, 
before the droop of $p$ on line 5,
if $D(i+a,j+b)$ is a 
blank tile (in which case the
condition on line 6 will be true), 
we remember its label $v$.
If the input $D(i,j)=\quo{\text{r}}$,
after the droop on line 5, $D(i,j)$ will become
a blank tile. We decorate it with the label
specified by input. 
Now instead of always choosing
the position of the r-tile of $p$ on row $i+a$,
we check the label $v$. If $v=\mathsf{x}$, 
we choose  the position of the r-tile of $p$ on
row $i+a$ as before, but if $v=\mathsf{-y}$,
we choose the position $(i',j+b)$ of the r-tile of
$p$ on column $j+b$. Note that this
construction guarantees that if the input
is $(D,f)$ where $D(i,j)=\text{``r''}$ and label $u$, and the output is $(D',f')$,
then $x_i\mon(D,f)=\mon(D',f')$ if $u=\mathsf{x}$,
and  $-y_j\mon(D,f)=\mon(D',f')$ if $u=\mathsf{-y}$.
The pseudocode for this modification is the following snippet,
and we replace lines 5--8 in Algorithm \ref{alg:forward} with it.
\vskip 0.15em
{\small
\begin{algorithm}[H]
\SetKwProg{Fn}{}{ } {}\SetKwFunction{Fun}{insert\_blank\_or\_resolve\_bump\_at\_r}%
\uIf{$D(i+a,j+b)$ is blank}{$v\gets f(i+a,j+b)$ }
Droop $p$ into $(i+a,j+b)$\;
\lIf{$D(i,j)$ is blank}{$f(i,j)\gets u$}
\uIf{$D(i+a,j+b)=\quo{\text{j}}$}
    {\uIf{$v=\mathsf{x}$}
         {$(i+a, j') \gets$ position of r-tile of $p$ on row $i+a$\;
         \Fun{$D,f,i+a, j', v$}}
     \uElse{$(i', j+b) \gets$ position of r-tile of $p$ on column $j+b$\;
         \Fun{$D,f,i', j+b, v$}}
    }
\end{algorithm}
}
\noindent (To be  pedantic, on line 20 of Algorithm \ref{alg:forward} we also need to pass the decoration as an argument, and the decoration must also
be returned.)

Similarly, we replace lines 5--10 in Algorithm
\ref{alg:backward} with the following snippet.
\vskip 0.15em
{\small
\begin{algorithm}[H]
\SetKwProg{Fn}{}{ } {}\SetKwFunction{Fun}{insert\_blank\_or\_resolve\_bump\_at\_j}%
\uIf{$D(i-a,j-b)$ is blank}{$v\gets f(i-a,j-b)$ }
Undroop $p$ into $(i-a,j-b)$\;
\lIf{$D(i,j)$ is blank}{$f(i,j)\gets u$}
\uIf{$D(i-a,j-b)=\quo{\text{r}}$}
    {\uIf{$v=\mathsf{x}$}
         {\uIf{$\forall j'>j-b$, $D(i-a,j')$ does not have a j-tile}{\Return{$D,f,v$}}
          $(i-a, j') \gets$ position of r-tile of $p$ on row $i-a$\;
         \Fun{$D,f,i-a, j', v$}}
     \uElse{
           {\uIf{$\forall i'>i-a$, $D(i',j-b)$ does not have a j-tile}{\Return{$D,f,v$}}
           $(i', j+b) \gets$ position of r-tile of $p$ on column $j+b$\;
           \Fun{$D,f,i', j-b, v$}}}
    }
\end{algorithm}}
Note that if the condition on
line 12 in the snippet above
is triggered, then
$p$ only turns once on
column $j-b$, which means that
$j-b=\pi(x)$. The rest of the
analysis for the modified 
algorithm is the same as before.
\end{proof}
Figure \ref{fig:examplexy} shows
an example for the same 
bumpless pipe dream as Figure
\ref{fig:example}, but with
labelled tiles.

\begin{figure}[H]
    \centering
    \begin{tikzpicture}
    \node[anchor=south west,inner sep=0] (image) at (-0.5,0)  
     {\includegraphics[trim=45 1350 500 45,clip, scale=0.1]{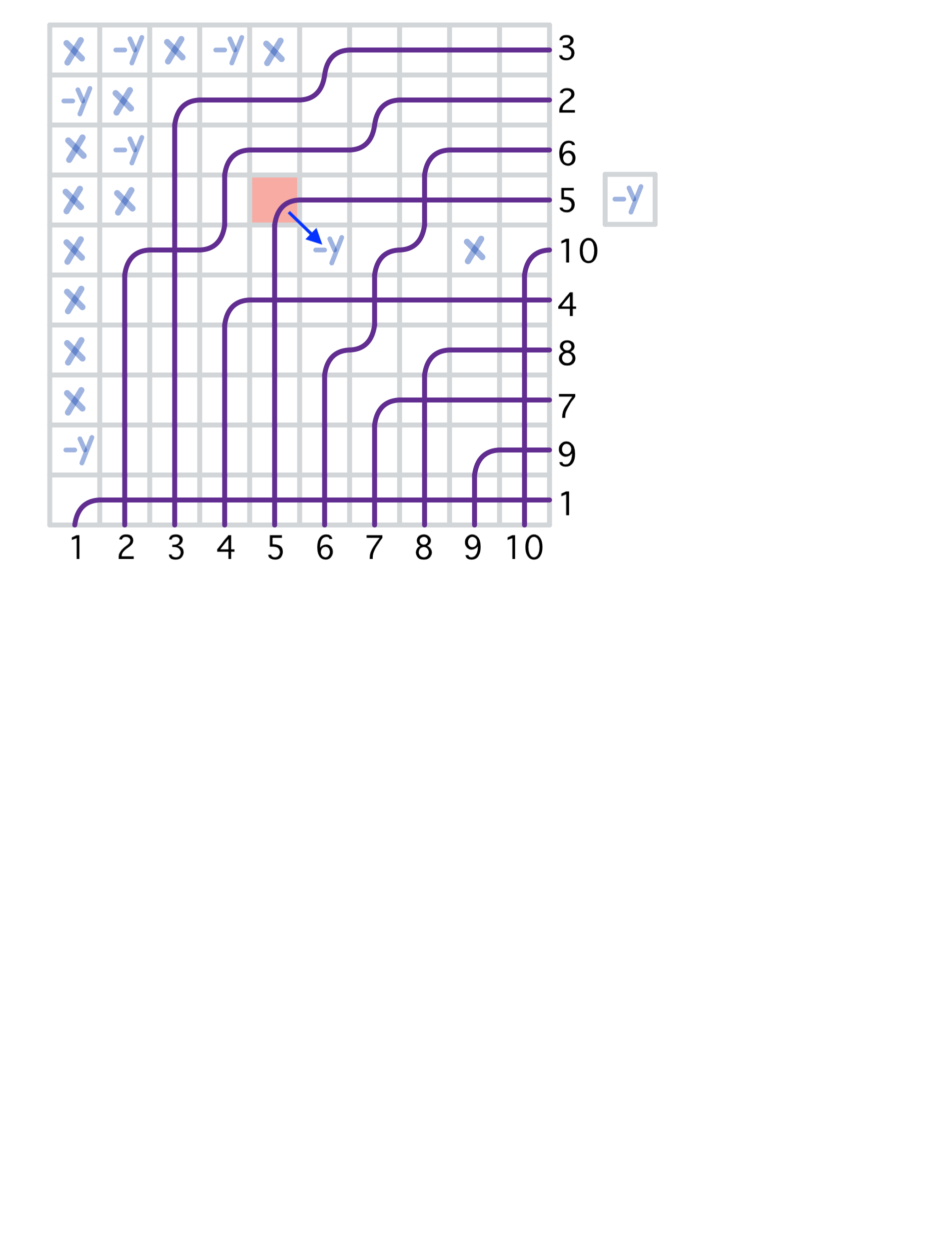}};
    \draw[black, thick, ->] (4.2,2) -- (4.8,2);
    \node[anchor=south west,inner sep=0] (image) at (5,0)  
     {\includegraphics[trim=45 1350 500 45,clip, scale=0.1]{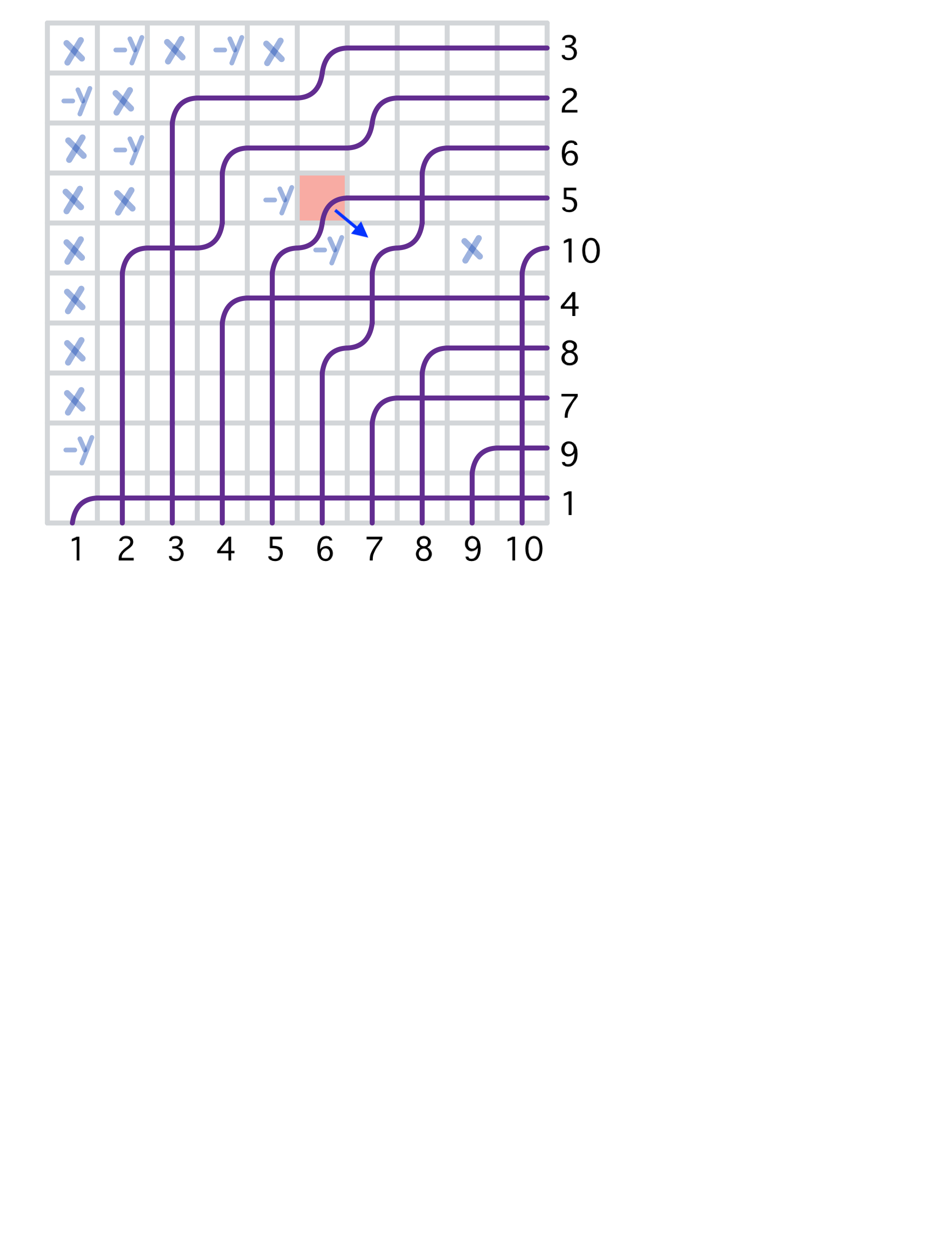}};
    \draw[black, thick, ->] (9.8,2) -- (10.3,2);
    \node[anchor=south west,inner sep=0] (image) at (10.5,0)
     {\includegraphics[trim=45 1350 500 45,clip, scale=0.1]{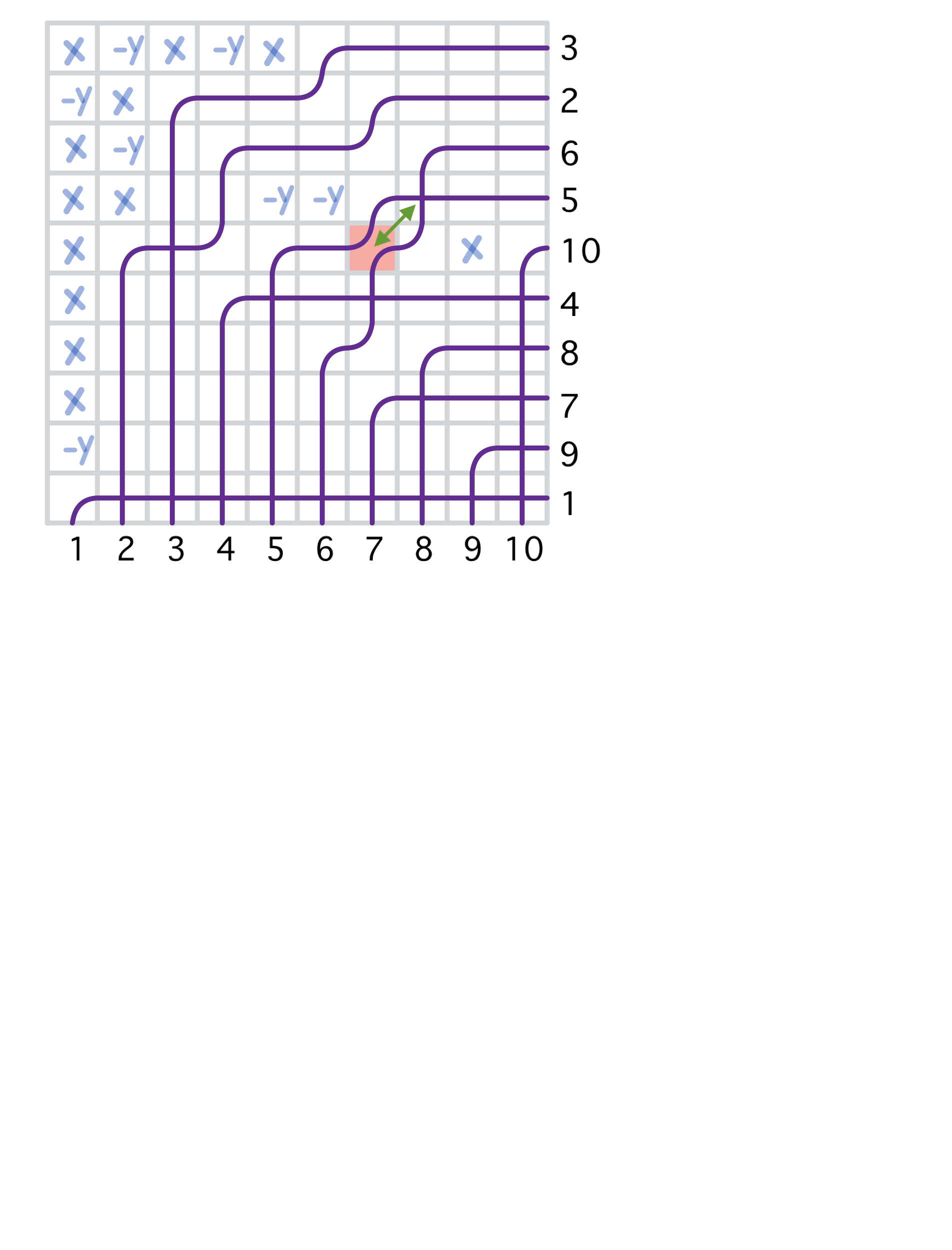}};
     
    \draw[black, thick, ->] (-1.2,-3) -- (-0.7,-3);
    \node[anchor=south west,inner sep=0] (image) at (-0.5,-5)
     {\includegraphics[trim=45 1350 500 45,clip, scale=0.1]{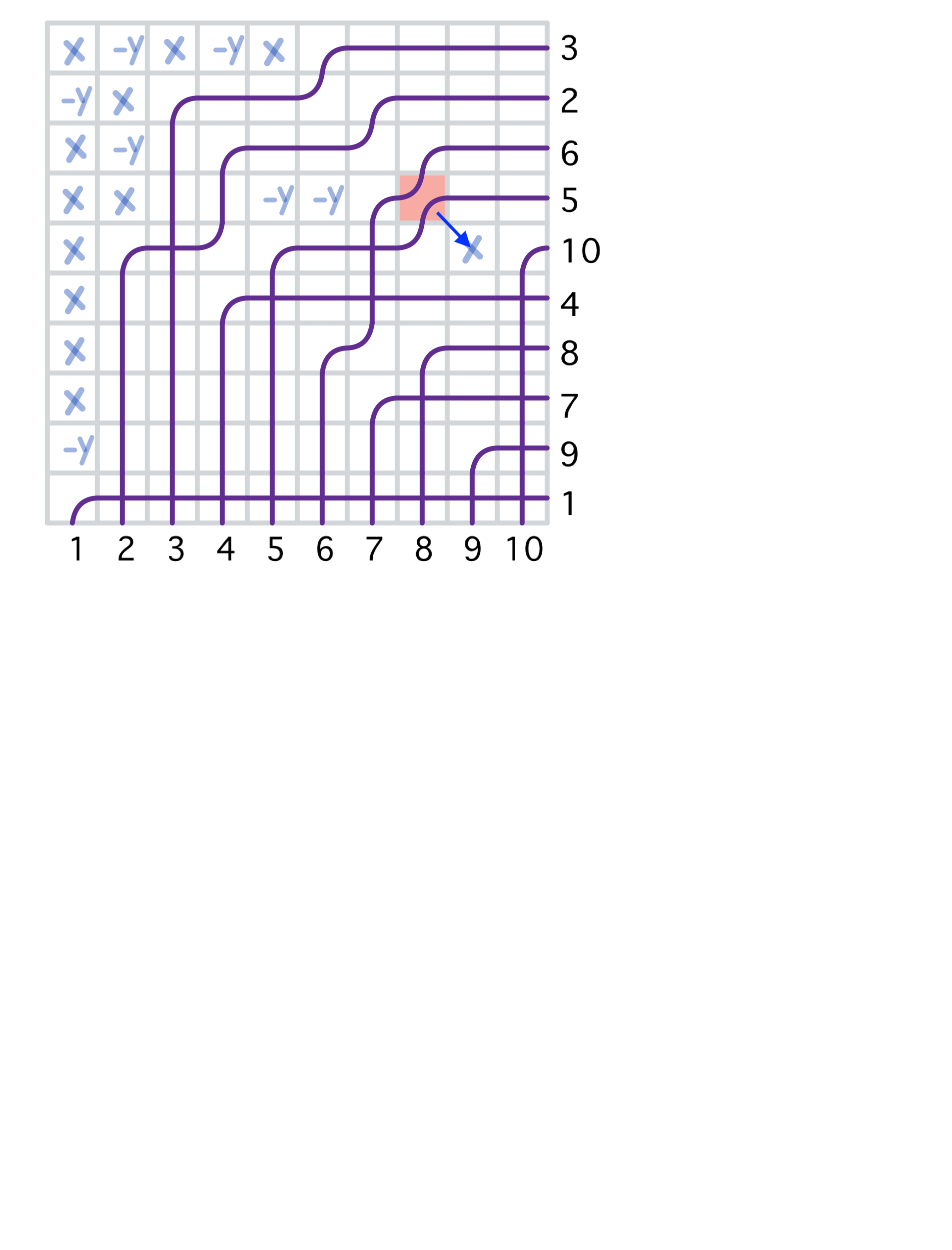}};
    \draw[black, thick, ->] (4.2,-3) -- (4.8,-3);
    \node[anchor=south west,inner sep=0] (image) at (5,-5)
     {\includegraphics[trim=45 1350 500 45,clip, scale=0.1]{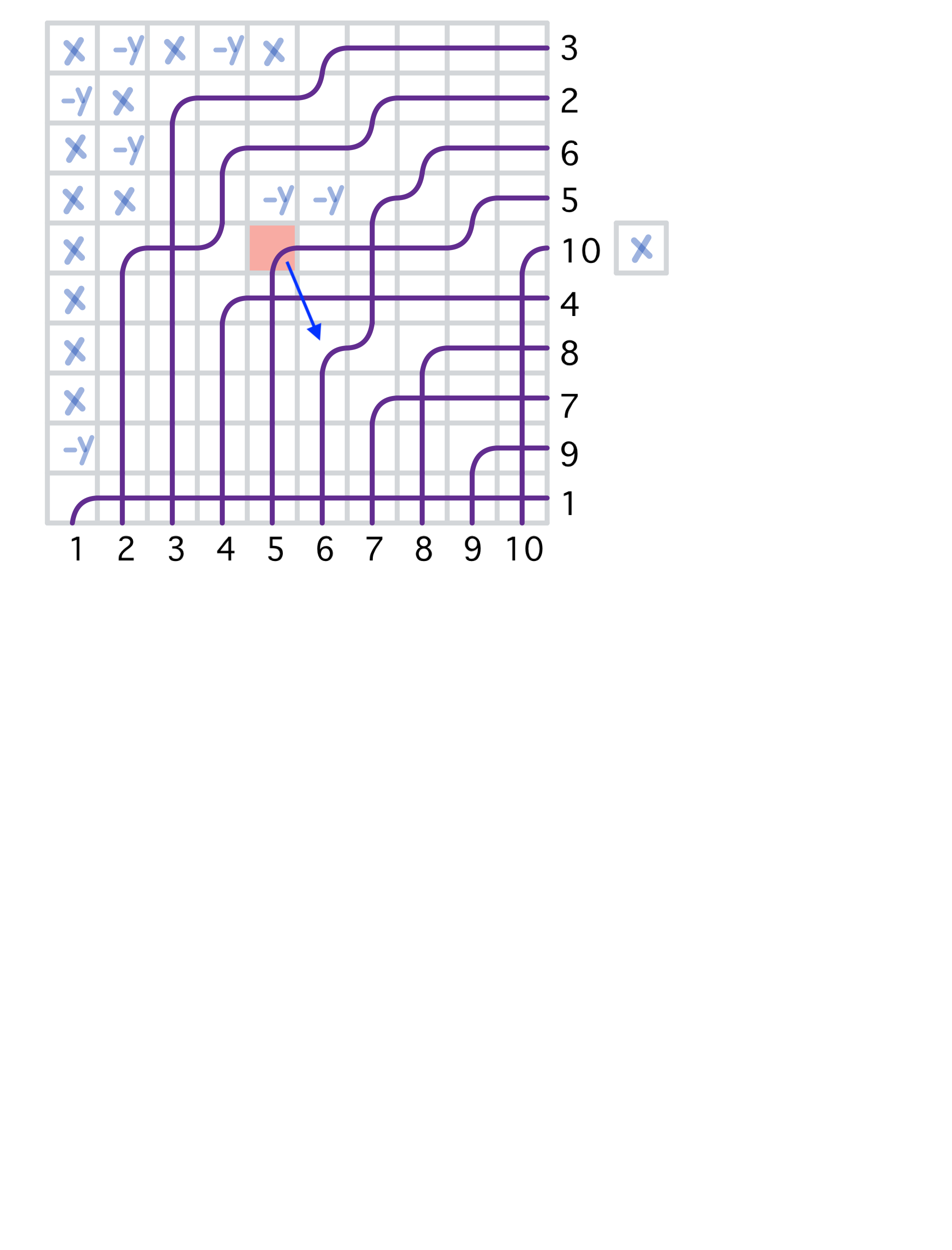}};
    \draw[black, thick, ->] (9.8,-3) -- (10.3,-3);
    \node[anchor=south west,inner sep=0] (image) at (10.5,-5)
     {\includegraphics[trim=45 1350 500 45,clip, scale=0.1]{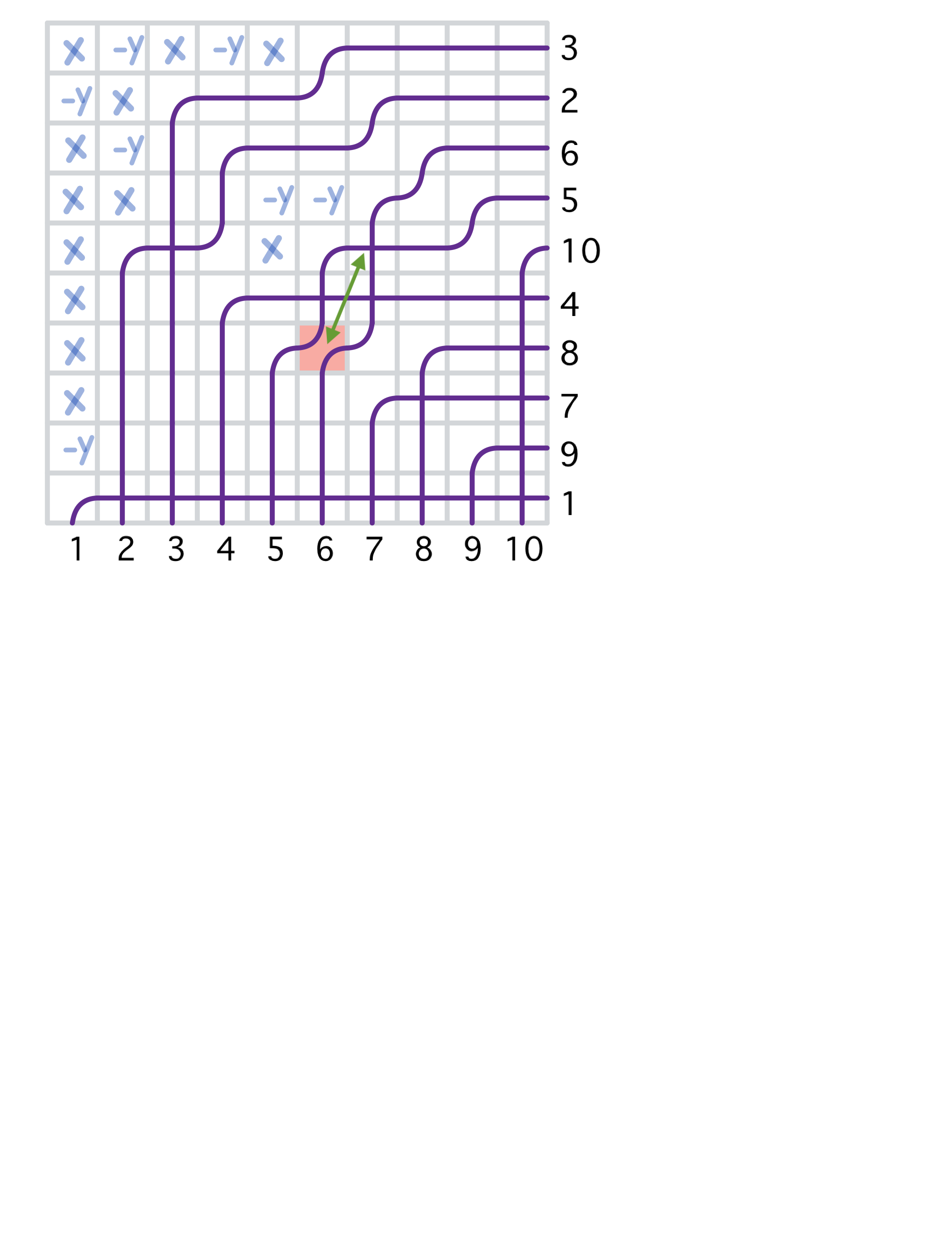}};
     
    \draw[black, thick, ->] (-1.2,-8) -- (-0.7,-8);
    \node[anchor=south west,inner sep=0] (image) at (0,-10)
     {\includegraphics[trim=45 1350 500 45,clip, scale=0.1]{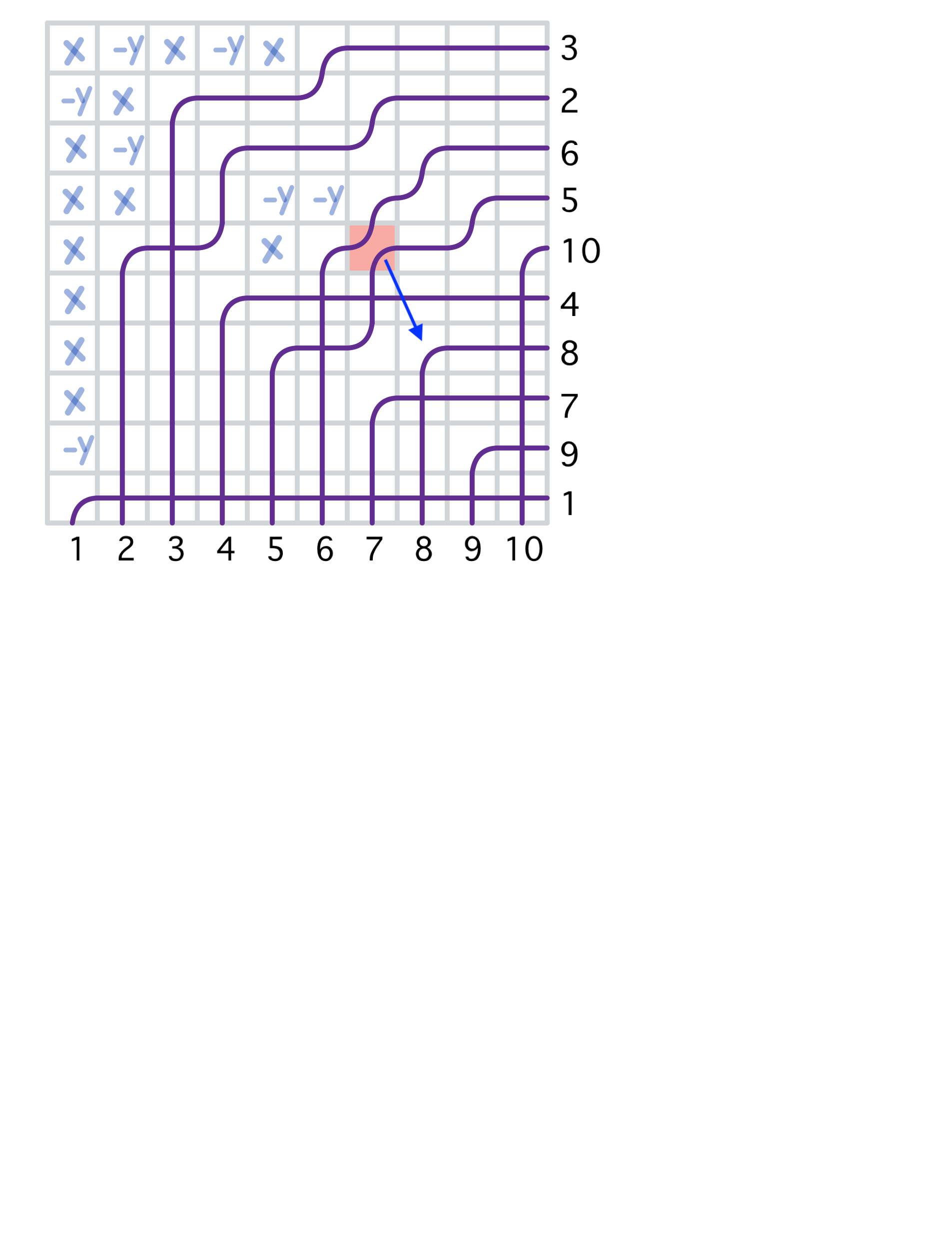}};
    \draw[black, thick, ->] (4.2,-8) -- (4.8,-8);
    \node[anchor=south west,inner sep=0] (image) at (5,-10)
     {\includegraphics[trim=45 1350 500 45,clip, scale=0.1]{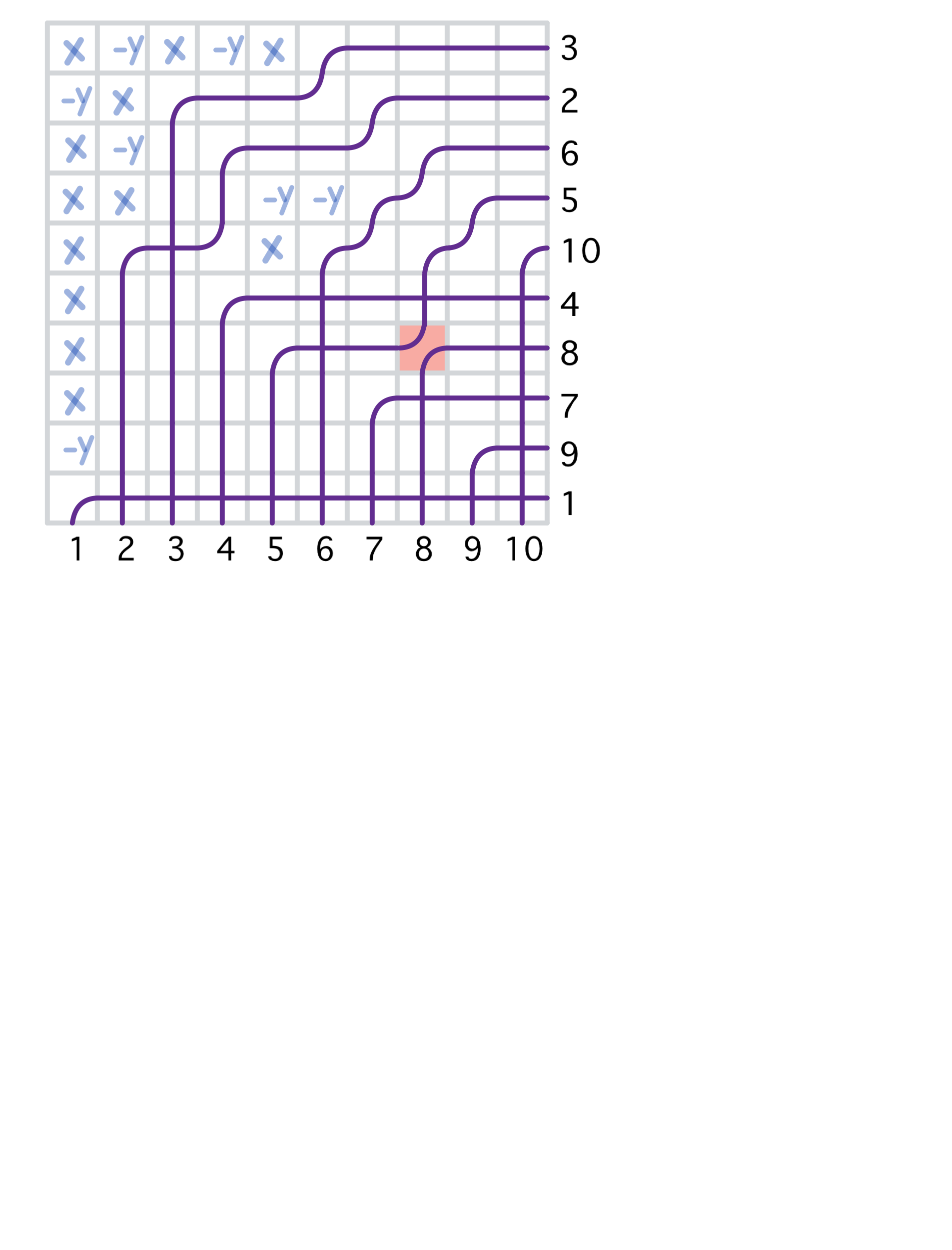}};
    \draw[black, thick, ->] (9.8,-8) -- (10.3,-8);
    \node[anchor=south west,inner sep=0] (image) at (10.5,-10)
     {\includegraphics[trim=45 1350 500 45,clip, scale=0.1]{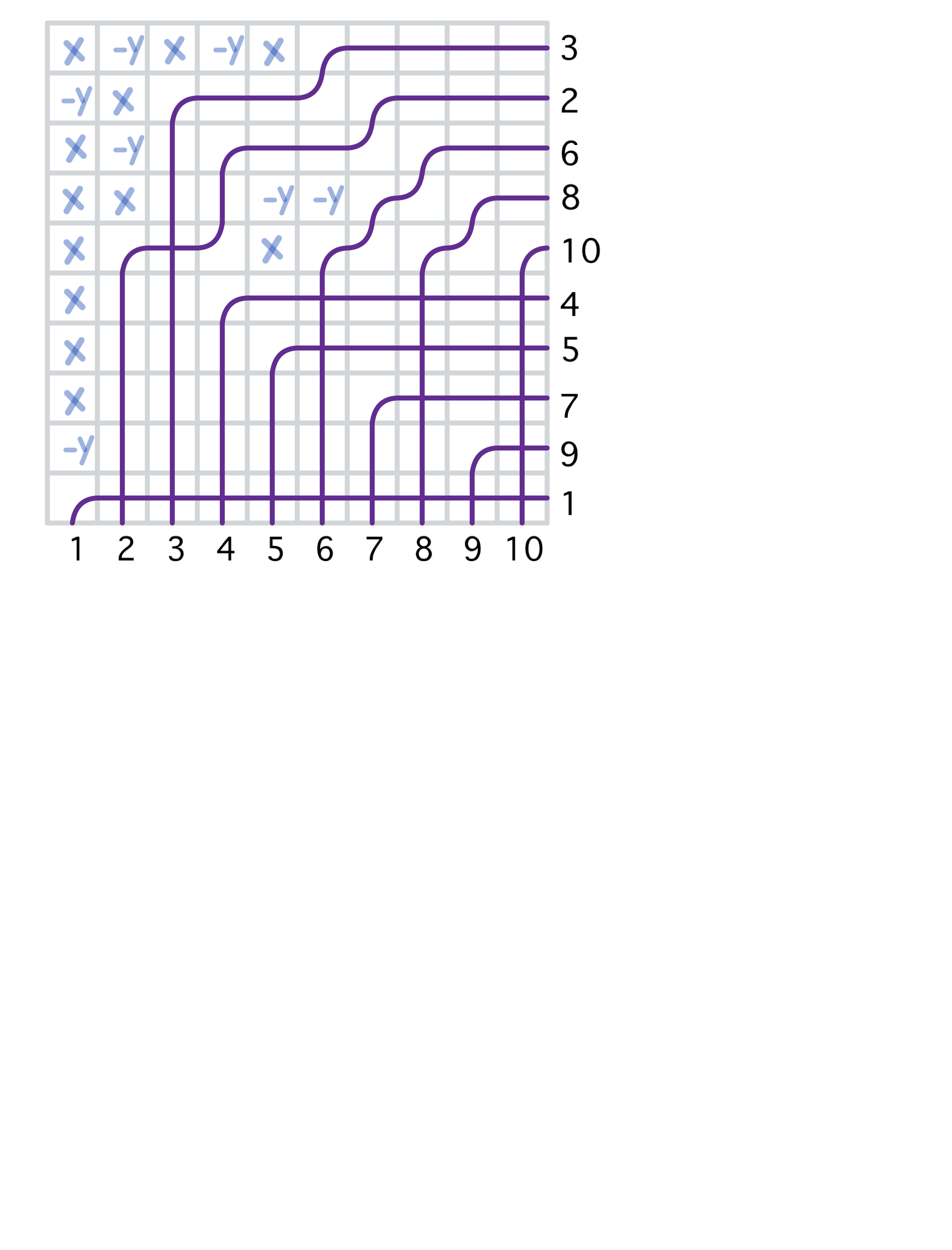}};
\end{tikzpicture}

      \caption{Insertion of a blank tile
      marked $\mathsf{-y}$ at
      $(i,j)=(4,5)$ for a decorated BPD of
      $\pi=[3,2,6,5,10,4,8,7,9,1]$}
    \label{fig:examplexy}

\end{figure}
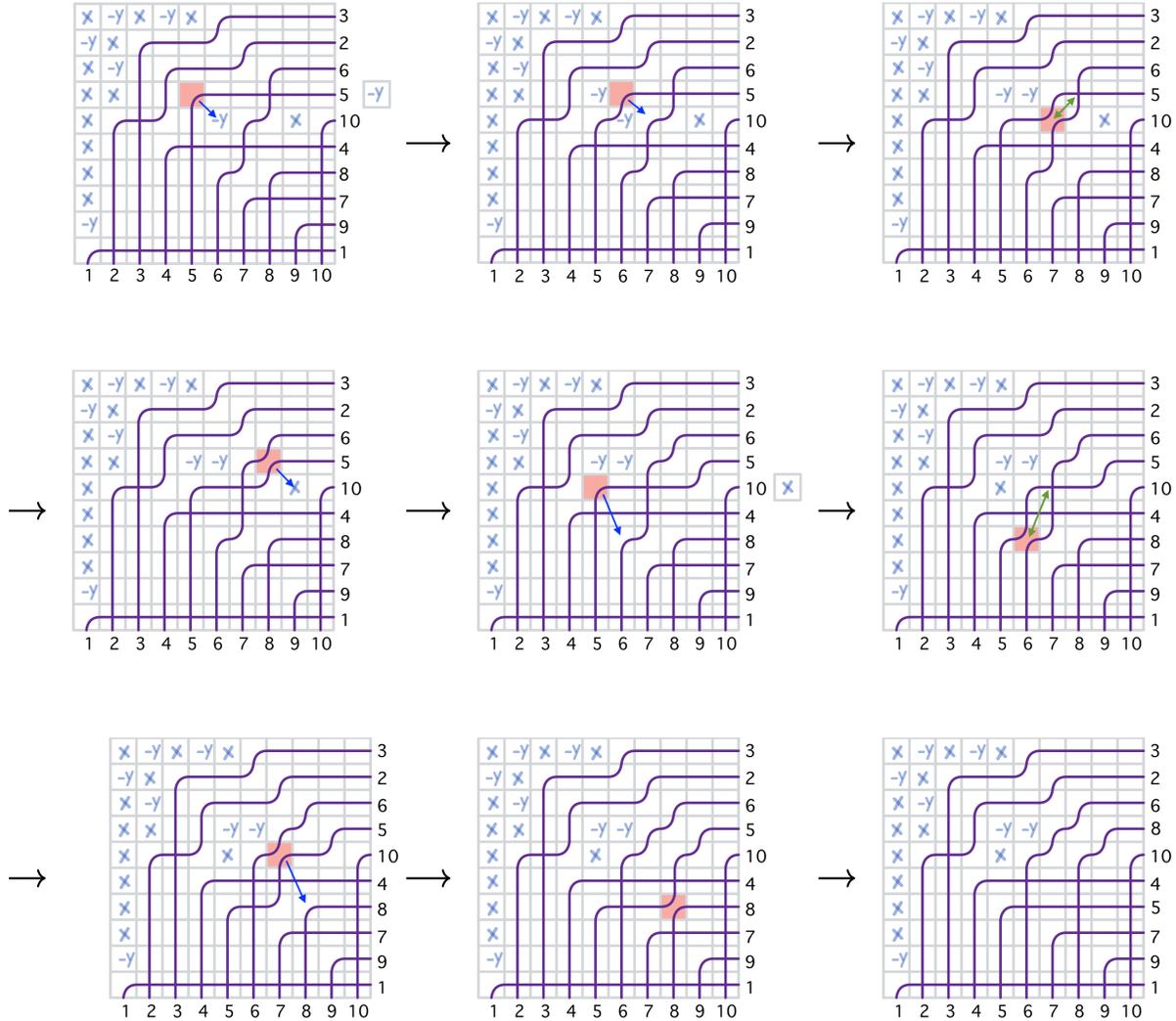
\ssection{Transition and Cotransition Formulas}

We discuss briefly the 
implication our results have
on the transition and cotransition formulas of (double) Schubert polynomials.
Transition and cotransition formulas are specializations
of formula (\ref{doublemonk}). If there is a unique 
$l>\alpha$ such that $\pi\, t_{\alpha, l}\gtrdot \pi$,
namely if the right side of  formula (\ref{doublemonk})
only has one summand, we get the \emph{transition formula}
for $\S_{\pi\,t_{\alpha, l}}$. In terms of 
combinatorial bijections, this is the simplest case because
only a single droop/undroop move is required to go 
between the bijection, and each move only modifies four 
tiles locally.
The details of this 
is given in \cite{weigandt2020bumpless}. Therefore, to
establish the transition formula alone for double Schubert
polynomials, we do not need to consider 
decorated bumpless pipe dreams. Billey, Holroyd, 
and Young gave a bijective proof for transition with
ordinary pipe dreams \cite{billey2019bijective}. There, the construction only works
for single Schubert polynomials. 

If, on the other hand, there is no $k<\alpha$ such that
$\pi\,t_{k,\alpha}$, we get the \emph{cotransition formula}.
Unlike the transition formula, if we only work with 
bumpless pipe dreams without decorations, we can only 
get the version for single Schubert polynomials. This
is analogous to the phenomenon in \cite{billey2019bijective}.
On the other hand, in \cite{knutson2019schubert}
a simple bijective proof of cotransition for double Schubert polynomials is given with ordinary pipe dreams, which
only requires changing one tile locally to go between
the bijection. As a direct consequence,
using the cotransition bijections for ordinary and
bumpless pipe dreams, we get a bijection of ordinary
and bumpless pipe dreams by reverse induction
on the length of the permutation. This idea is similar
to the approach in \cite{fan2018bumpless} where a shape preserving bijection between reduced word tableaux for a permutation $w$ and Edelman-Greene pipe dreams of $w$ is constructed.

\bibliographystyle{alpha}
\bibliography{ref}
\end{document}